\documentclass[12pt]{article}
\usepackage{amssymb,amscd,amsthm,amsmath,latexsym, amstext,mathrsfs,verbatim,longtable,array,multirow,mathtools}
\usepackage{enumitem}
\newlist{Case}{enumerate}{2} % 创建两层嵌套列表
\setlist[Case, 1]{
	label=\textbf{Case \arabic*.}, % 主层级编号格式
	leftmargin=4em,
}
\setlist[Case, 2]{
	label=\textbf{Subcase \arabic{Casei}.\arabic*}, % 子层级编号格式
	leftmargin=6em,
}
\setlist[Case, 3]{
	label=\textbf{Subcase \arabic{Casei}.\arabic*}, % 子层级编号格式
	leftmargin=6em,
}
\setlist[Case, 4]{
	label=\textbf{Subcase \arabic{Casei}.\arabic*}, % 子层级编号格式
	leftmargin=6em,
}

\def\lijntje{\vrule height2.4pt depth-2pt width0.5in}

\def\vlijntje{\vrule height0.45in depth0.4pt width0.4pt}
\def\vlijn{\buildrel {\hbox to 0pt{\hss$\textstyle\circ$\hss}}\over\vlijntje}

\def\dlijntje{{\vrule height2pt depth-1.6pt
		width0.5in}\llap{\vrule height4pt depth-3.6pt width0.5in}}
\def\tlijntje{{\vrule height1.7pt depth-1.3pt
		width0.5in}\llap{\vrule height3.0pt depth-2.6pt width0.5in}\llap{\vrule height4.3pt depth-3.9pt width0.5in}
}
\def\vtriple#1\over#2\over#3{\mathrel{\mathop{\kern0pt #2}\limits_{\hbox
			to 0pt{\hss$#1$\hss}}^{\hbox to 0pt{\hss$#3$\hss}}}}
\def\rvtriple#1\over#2\over#3{\mathrel{\mathop{\kern0pt #2}\limits_{\hbox
			to 0pt{\hss$#3$\hss}}^{\hbox to 0pt{\hss$#1$\hss}}}}
\def\Ct{\vtriple{\scriptstyle 2}\over\circ\over{}
	\kern-1pt\lijntje\kern-1pt\vtriple{\scriptstyle 1}\over\circ\over{}
	\kern-4pt{\dlijntje \kern -25pt<}\kern8pt
	\vtriple{\scriptstyle 0}\over\circ\over{}\kern-1pt
}
\def\Bt{\vtriple{\scriptstyle 2}\over\circ\over{}
	\kern-1pt\lijntje\kern-1pt\vtriple{\scriptstyle 1}\over\circ\over{}
	\kern-4pt{\dlijntje \kern -25pt>}\kern8pt
	\vtriple{\scriptstyle 0}\over\circ\over{}\kern-1pt}

\def\det{{\rm det}}

\newcommand{\C}{\mathbb C}

\def\Dm{\vtriple{\scriptstyle n+1}\over\circ\over{}\kern-1pt\lijntje\kern-1pt
	\vtriple{\scriptstyle{n}}\over\circ\over{}
	\cdots\cdots\vtriple{\scriptstyle 4}\over\circ\over{}\kern-1pt\lijntje\kern-1pt
	\vtriple{\scriptstyle 3}\over\circ\over{\buildrel
		{\scriptstyle 2}\over\vlijn}\kern-1pt\lijntje\kern-1pt
	\vtriple{1}\over\circ\over{}\kern-1pt}

\def\Dn{\vtriple{\scriptstyle n}\over\circ\over{}\kern-1pt\lijntje\kern-1pt
	\vtriple{\scriptstyle{n-1}}\over\circ\over{}
	\cdots\cdots\vtriple{\scriptstyle 4}\over\circ\over{}\kern-1pt\lijntje\kern-1pt
	\vtriple{\scriptstyle 3}\over\circ\over{\buildrel
		{\scriptstyle 2}\over\vlijn}\kern-1pt\lijntje\kern-1pt
	\vtriple{1}\over\circ\over{}\kern-1pt}
\def\En{\vtriple{\scriptstyle n}\over\circ\over{}\kern-1pt\lijntje\kern-1pt
	\vtriple{\scriptstyle{n-1}}\over\circ\over{}
	\cdots\cdots\vtriple{\scriptstyle 5}\over\circ\over{}\kern-1pt\lijntje\kern-1pt
	\vtriple{\scriptstyle 4}\over\circ\over{\buildrel
		{\scriptstyle 2}\over\vlijn}\kern-1pt\lijntje\kern-1pt
	\vtriple{\scriptstyle 3}\over\circ\over{}\kern-1pt\lijntje\kern-1pt
	\vtriple{\scriptstyle 1}\over\circ\over{}\kern-1pt}
\def\An{\vtriple{\scriptstyle n}\over\circ\over{}\kern-1pt\lijntje\kern-1pt
	\vtriple{\scriptstyle{n-1}}\over\circ\over{}\kern-1pt\lijntje\kern-1pt
	\vtriple{\scriptstyle n-2}\over\circ\over{}
	\cdots\cdots
	\vtriple{\scriptstyle 2}\over\circ\over{}\kern-1pt\lijntje\kern-1pt
	\vtriple{\scriptstyle 1}\over\circ\over{}\kern-1pt}
\def\Cn{\vtriple{\scriptstyle n-1}\over\circ\over{}
	\kern-1pt\lijntje\kern-1pt\vtriple{\scriptstyle{n-2}}\over\circ\over{}
	\cdots\cdots
	\vtriple{\scriptstyle 2}\over\circ\over{}
	\kern-1pt\lijntje\kern-1pt\vtriple{\scriptstyle 1}\over\circ\over{}
	\kern-4pt{\dlijntje \kern -25pt<}\kern10pt
	\vtriple{\scriptstyle 0}\over\circ\over{}\kern-1pt}
\def\Ct{\vtriple{\scriptstyle 2}\over\circ\over{}
	\kern-1pt\lijntje\kern-1pt\vtriple{\scriptstyle 1}\over\circ\over{}
	\kern-4pt{\dlijntje \kern -25pt<}\kern12pt
	\vtriple{\scriptstyle 0}\over\circ\over{}\kern-1pt
}
\def\Bn{\vtriple{\scriptstyle n-1}\over\circ\over{}
	\kern-1pt\lijntje\kern-1pt\vtriple{\scriptstyle{n-2}}\over\circ\over{}
	\cdots\cdots
	\vtriple{\scriptstyle 2}\over\circ\over{}
	\kern-1pt\lijntje\kern-1pt\vtriple{\scriptstyle 1}\over\circ\over{}
	\kern-4pt{\dlijntje \kern -25pt>}\kern10pt
	\vtriple{\scriptstyle 0}\over\circ\over{}\kern-1pt}
\def\Bt{\vtriple{\scriptstyle 2}\over\circ\over{}
	\kern-1pt\lijntje\kern-1pt\vtriple{\scriptstyle 1}\over\circ\over{}
	\kern-4pt{\dlijntje \kern -25pt>}\kern12pt
	\vtriple{\scriptstyle 0}\over\circ\over{}\kern-1pt}
\def\Es{\vtriple{\scriptstyle 6}\over\circ\over{}\kern-1pt\lijntje\kern-1pt
	\vtriple{\scriptstyle 5}\over\circ\over{}\kern-1pt\lijntje\kern-1pt
	\vtriple{\scriptstyle 4}\over\circ\over{\buildrel
		{\scriptstyle 2}\over\vlijn}\kern-1pt\lijntje\kern-1pt
	\vtriple{3}\over\circ\over{}\kern-1pt\lijntje\kern-1pt
	\vtriple{\scriptstyle 1}\over\circ\over{}\kern-1pt}
\def\Ff{
	\vtriple{\scriptstyle 1}\over\circ\over{}
	\kern-1pt\lijntje\kern-1pt\vtriple{\scriptstyle 2}\over\circ\over{}
	\kern-4pt{\dlijntje \kern -25pt<}\kern10pt
	\vtriple{\scriptstyle 3}\over\circ\over{}\kern-1pt\lijntje\kern-1pt
	\vtriple{\scriptstyle 4}\over\circ\over{}
	\kern-1pt}
\def\Ht{
	\vtriple{\scriptstyle 1}\over\circ\over{}
	\kern-1pt\overset{5}{\lijntje}\kern-1pt\vtriple{\scriptstyle 2}\over\circ\over{}
	\kern-1pt\lijntje\kern-1pt
	\vtriple{\scriptstyle 3}\over\circ\over{}\kern-1pt}
\def\Hf{
	\vtriple{\scriptstyle 1}\over\circ\over{}
	\kern-1pt\overset{5}{\lijntje}\kern-1pt\vtriple{\scriptstyle 2}\over\circ\over{}
	\kern-1pt\lijntje\kern-1pt
	\vtriple{\scriptstyle 3}\over\circ\over{}\kern-1pt\lijntje\kern-1pt
	\vtriple{\scriptstyle 4}\over\circ\over{}
	\kern-1pt}
\def\In{
	\vtriple{\scriptstyle 0}\over\circ\over{}
	\kern-1pt\overset{n}{\lijntje}\kern-1pt\vtriple{\scriptstyle 1}\over\circ\over{}
	\kern-1pt}
\def\Gt{
	\vtriple{\scriptstyle 0}\over\circ\over{}
	\kern-4pt{\tlijntje\kern -25pt<}\kern 10pt\vtriple{\scriptstyle 1}\over\circ\over{}
	\kern-1pt}
\def\EBn{\vtriple{\scriptstyle n-1}\over\circ\over{}
	\kern-1pt\lijntje\kern-1pt\vtriple{\scriptstyle{n-2}}\over\circ\over{\buildrel
		{\scriptstyle -1}\over\vlijn}\cdots\cdots
	\vtriple{\scriptstyle 2}\over\circ\over{}
	\kern-1pt\lijntje\kern-1pt\vtriple{\scriptstyle 1}\over\circ\over{}
	\kern-4pt{\dlijntje \kern -25pt<}\kern8pt
	\vtriple{\scriptstyle 0}\over\circ\over{}\kern-1pt}
\def\Cn{\vtriple{\scriptstyle n-1}\over\circ\over{}
	\kern-1pt\lijntje\kern-1pt\vtriple{\scriptstyle{n-2}}\over\circ\over{}
	\cdots\cdots
	\vtriple{\scriptstyle 2}\over\circ\over{}
	\kern-1pt\lijntje\kern-1pt\vtriple{\scriptstyle 1}\over\circ\over{}
	\kern-4pt{\dlijntje \kern -25pt<}\kern10pt
	\vtriple{\scriptstyle 0}\over\circ\over{}\kern-1pt}
\def\ECn{\vtriple{\scriptstyle -2}\over\circ\over{}
	\kern-4pt{\dlijntje \kern -25pt>}\kern8pt\vtriple{\scriptstyle n-1}\over\circ\over{}
	\kern-1pt\lijntje\kern-1pt\vtriple{\scriptstyle{n-2}}\over\circ\over{}
	\cdots\cdots
	\vtriple{\scriptstyle 2}\over\circ\over{}
	\kern-1pt\lijntje\kern-1pt\vtriple{\scriptstyle 1}\over\circ\over{}
	\kern-4pt{\dlijntje \kern -25pt<}\kern12pt
	\vtriple{\scriptstyle 0}\over\circ\over{}\kern-1pt}
\def\Fo{\vtriple{\scriptstyle -1}\over\circ\over{}
	\kern-1pt\lijntje\kern-1pt
	\vtriple{\scriptstyle 1}\over\circ\over{}
	\kern-1pt\lijntje\kern-1pt\vtriple{\scriptstyle 2}\over\circ\over{}
	\kern-4pt{\dlijntje \kern -25pt<}\kern8pt
	\vtriple{\scriptstyle 3}\over\circ\over{}\kern-1pt\lijntje\kern-1pt
	\vtriple{\scriptstyle 4}\over\circ\over{}
	\kern-1pt}
\def\Ft{
	\vtriple{\scriptstyle 1}\over\circ\over{}
	\kern-1pt\lijntje\kern-1pt\vtriple{\scriptstyle 2}\over\circ\over{}
	\kern-4pt{\dlijntje \kern -25pt<}\kern8pt
	\vtriple{\scriptstyle 3}\over\circ\over{}\kern-1pt\lijntje\kern-1pt
	\vtriple{\scriptstyle 4}\over\circ\over{}
	\kern-1pt\lijntje\kern-1pt
	\vtriple{\scriptstyle -2}\over\circ\over{}
	\kern-1pt}
\def\Go{\vtriple{\scriptstyle -1}\over\circ\over{}
	\kern-1pt\lijntje\kern-1pt
	\vtriple{\scriptstyle 0}\over\circ\over{}
	\kern-4pt{\tlijntje\kern -25pt<}\kern 12pt\vtriple{\scriptstyle 1}\over\circ\over{}
	\kern-1pt}
\def\Gf{
	\vtriple{\scriptstyle 0}\over\circ\over{}
	\kern-4pt{\tlijntje\kern -25pt<}\kern 12pt\vtriple{\scriptstyle 1}\over\circ\over{}
	\kern-1pt\lijntje\kern-1pt
	\vtriple{\scriptstyle -2}\over\circ\over{}
	\kern-1pt}

\numberwithin{equation}{section}

\newtheorem{lemma}{Lemma}[section]
\newtheorem{cor}[lemma]{Corollary}
\newtheorem{prop}[lemma]{Proposition}
\newtheorem{thm}[lemma]{Theorem}

\theoremstyle{remark}
\newtheorem{rem}[lemma]{Remark}

\theoremstyle{definition}

\newtheorem{defn}[lemma]{Definition}

\begin{document}
	\title{The characteristic polynomials of imprimitive groups and affine Coxeter groups  }
	\author{ Chenyue Feng, Shoumin Liu\footnote{The author is  funded by the NSFC (Grant No. 11971181, Grant No.12271298) },   Xumin Wang}
	\date{}
	\maketitle
	%\mainmatter
	%\setcounter{section}{-1} \tableofcontents
	\begin{abstract}
		In this paper, we will seek appropriate generators to define the characteristic polynomials of $G(r,1,n)$, and prove that for every finite dimensional representation of $G(r,1,n)$, the characteristic polynomial of $G(r,1,n)$ determines the character of this representation. Furthermore, the same conclusion holds for affine Coxeter groups \(\widetilde{W}\) and $G(r,p,n)$.
	\end{abstract}

	\section{Introduction}

\hspace{1.5em}Given a finite group $G = \{1, g_1, \cdots, g_n\}$ and $\lambda_G$ being the left regular representation of $G$, Dedekind studied the determinant
	\begin{equation}\label{1.1}
		f_{\lambda_G}(z) = det(x_0I + x_1\lambda_G(g_1) + \cdots + x_n\lambda_G(g_n))
    \end{equation}
for some groups and  its  factorizations with $z=(x_0, x_1, \dots, x_n)$ in \cite{D1969}.
It is natural that  one may define the corresponding polynomial $f_{\pi}(z)$ associated to a general finite dimensional representation $\pi$ of $G$. Let $\hat{G}$ be the set of equivalence classes of irreducible unitary representations of $G$, and $d_{\pi}$ be the dimension of the representation $\pi$. Later in 1896 Frobenius proved the following theorem in \cite{F1896}.

\begin{thm}\label{0}
	If $G = \{1, g_1, \cdots, g_n\}$ is a finite group, then
	$$f_{\lambda_G}(z)=\prod_{\pi\in\hat{G}}(f_{\pi}(z))^{d_\pi} .$$
	Moreover, each $f_{\pi}(z)$ is an irreducible polynomial of degree $d_{\pi}$.
\end{thm}

The set of zeros of a polynomial  $f_{\pi}(z)$ was called the group determinant of $G$, and many related
works by Dedekind and Frobenius are inspired by group representation theory to \cite{Cu2000,D1969,Di1902,Di1921,Di1975,F1896,HKV2018,KV2012} of this topic.
It was later proved by Formanek and Sibley in \cite{FS1991} that the group determinant also determines the group up to isomorphism,
showing that the group determinant itself is rich enough to make the correspondence.

Analogous to Theorem (\ref{1.1}), we consider a more general situation.
Given $n$ square $N\times N$ matrices, the determinant of their linear combination,
%with undetermined coefficients,
$det(x_1A_1+\cdots +x_nA_n)$, is called the characteristic polynomial of degree $N$ in variables $x_1,...,x_n$,  Zeros of this polynomial form an algebraic manifold in the projective space $\C{\mathbb P}^{n-1}$. This manifold is called the determinantal manifold or determinantal hypersurface for the tuple $(A_1,\dots,A_n)$. We use the following notation
$$\sigma(A_1,...,A_n)=
\Big\{[x_1:\cdots :x_n]\in \C{\mathbb P}^{n-1}: \ det(x_1A_1+\cdots +x_nA_n)=0\Big\}.$$
However, the characteristic polynomials for several general matrices are a new frontier in linear algebra. In \cite{Y2009}, the notion $\sigma(A_1,...,A_n)$ is defined as the projective spectrum of operators  by  Yang through the multiparameter pencil $x_1A_1+\cdots +x_nA_n$, and  multivariable homogeneous characteristic polynomials have been studied there.
Fruitful results have been obtained in \cite{GY2017, GR2014, HY2023, HY2024,  HY2018,S2024}.
A wealth of conclusions have been reached regarding the same topics for finite dimensional complex simple Lie algebra $\mathfrak{g}$.
There is a correspondence between finite dimensional representations of $\mathfrak{g}$ and their characteristic polynomials.
Meanwhile, the invariance and symmetry properties of the characteristic polynomial have been thoroughly investigated and well-established.
For the detailed information regarding these contents, it can be found in
\cite{CCD2019,FLW2024,GLW2024,HY2018,HZ2019,H2021,KKSW,KY2021}.

The group algebra for a group contains most of the information of its representations. Instead of all the group elements, whether taking only the appropriate generators for the group will also reflect the information of group representations and to what extent the representation information can be reflected? This is a quite natural question. Hence, analogously, we have the following definition.
\begin{defn}
		Let $G$ be a group with a finite generating set $S=\{s_1, s_2, \dots, s_n\}$ and $rep(G)$ denote the set of all finite dimensional complex linear representations of $G$. For any $\rho \in rep(G)$, the \textit{characteristic polynomial of} $\rho$ is \[d(S,\rho)=det(x_0I+x_1\rho(s_1)+\cdots+x_{n-1}\rho(s_{n-1})+x_n\rho(s_n)).\]
\end{defn}

When $G$ is a finite Coxeter group, in \cite{CS2021}, \v{C}u\v{c}kovi\v{c}, Stessin and Tchernev
 show that if $\rho$ is the left regular representation, then
$D(S,\rho)$ determines the isomorphism class of $G$.
 If $G$ is not of exceptional types,
and  $\rho$ is any finite dimensional representation, then
$D(S,\rho)$ determines $\rho$. Peebles and Stessin came to the same conclusion about affine Coxeter groups of type $\widetilde{A}$ in \cite{PS2024}.
This result was established by constructing an Abelian normal subgroup of the canonical Coxeter affine group \(\tilde{A}_n\), with the core issue lying in the choice of generators for the affine Coxeter group. It is well-known that the affine Coxeter group \(\widetilde{W}\) can be regarded as a semidirect product $W \ltimes T$ of the root lattice $T$ and the corresponding finite Coxeter group $W$. By constructing the generating set \(A = \{g_1, \ldots, g_k, \alpha_1, \ldots, \alpha_n, \alpha_1^{-1}, \ldots, \alpha_n^{-1}\}\) through this structure, we will show that the same conclusion holds for the affine Coxeter groups of all irreducible types.

The sketch of this paper is as follows.
	Preliminary results of imprimitive group $G(r,1,n)$ are presented in Section 2 for preparations. In Section 3, a correspondence between the characteristic polynomials and the representations of $G(r,1,n)$ is established through the echelon form. In Section 4, we define the characteristic polynomials for $\widetilde{W}$ and $G(r,p,n)$, and obtain the similar conclusion.

	\section{Signatures of words in $G(r,1,p)$}
	\hspace{1.3em}we first present the definition of $G(r,1,p)$ from \cite{L2017}.

   \begin{defn}The imprimitive group $G(r,1,n)$ is defined with generators $s_1,\cdots,s_{n-1},t_1,\cdots,t_n$, subject to the following relations.
   	\begin{flushleft}
   		\begin{enumerate}[label=(\alph*).]
   	\item	$s_is_j=s_js_i,\,\lvert i-j\rvert >1,$
   	\item	$s_is_{i+1}s_i=s_{i+1}s_{i}s_{i+1},\,i\le {n-2},$
   	\item	$s_{n-1}t_ns_{n-1}t_n=t_ns_{n-1}t_ns_{n-1},$
   	\item	$s_i t_n=t_n s_i,\,i \le {n-2},$
   	\item	$s_i^2=1=t_j^r,i=1,\cdots,n-1,\,j=1,\cdots,n,$
   	\item	$t_i=s_i\cdots s_{n-1}t_n s_{n-1} \cdots s_i, i \le {n-1}.$	
   	\end{enumerate}
   \end{flushleft}
  \end{defn}
   According to the defining relations, it is not difficult to verify the follow proposition, which can be used in the below.
   \begin{prop}
   	Let $1 \le i \le {n-1}$ and $1 \le k\le n$. In $G(r,1,p)$, we can have the following equations.
   	\begin{flushleft}
   	\begin{enumerate}[label=(\roman*).]
   		\item  $s_it_i=t_{i+1}s_i,i\le n-1,$ \label{i}
   	    \item $s_it_{i+1}=t_is_i,i\le n-1,$  \label{ii}
   		\item $s_it_k=t_ks_i, k\le i-1 \, \text{or} \,\, k\ge i+2,$ \label{(iii)}
   		\item $t_i t_j=t_j t_i.$ \label{(iv)}
   	\end{enumerate}
   	\end{flushleft}
   \end{prop}

	\begin{proof} \ref{i} If $i \le {n-1}$,\; we have
		\[s_i t_i=s_i s_i s_{i+1}\cdots s_{n-1} t_n s_{n-1} \cdots s_{i+1} s_i =s_{i+1}\cdots s_{n-1} t_n s_{n-1} \cdots s_{i+1} s_i=t_{i+1}s_i.\]
		Similar to (i), the (ii) can be proved.
		
		 \ref{(iii)} If $k \le {i-1}$, we have
		\begin{align*}
			s_i t_k& \stackrel{(f)}{=}s_i s_k \cdots s_{i-1}s_i \cdots s_{n-1}t_n s_{n-1} \cdots s_i s_{i-1} \cdots s_k\\
			&\stackrel{(a)}{=}s_k \cdots s_i s_{i-1}s_i \cdots s_{n-1}t_n s_{n-1} \cdots s_i s_{i-1} \cdots s_k \\
			&\stackrel{(b)}{=}s_k \cdots s_{i-1} s_i s_{i-1} \cdots s_{n-1}t_n s_{n-1} \cdots s_i s_{i-1} \cdots s_k \\
			&\stackrel{(a)}{=}s_k \cdots s_{i-1} s_i  \cdots s_{n-1} s_{i-1}t_n s_{n-1} \cdots s_i s_{i-1} \cdots s_k \\
			&\stackrel{(d)}{=}s_k \cdots s_{i-1} s_i  \cdots s_{n-1} t_n s_{i-1} s_{n-1} \cdots s_i s_{i-1} \cdots s_k \\
			&\stackrel{(a)}{=}s_k \cdots s_{i-1} s_i  \cdots s_{n-1} t_n  s_{n-1} \cdots s_{i-1} s_i s_{i-1} \cdots s_k \\
			&\stackrel{(b)}{=}s_k \cdots s_{i-1} s_i  \cdots s_{n-1} t_n  s_{n-1} \cdots s_{i} s_{i-1} s_{i} \cdots s_k \\
			&\stackrel{(a)}{=}s_k \cdots s_{i-1} s_i  \cdots s_{n-1} t_n  s_{n-1} \cdots s_{i} s_{i-1}  \cdots s_k s_i \\
			&\stackrel{(f)}{=}t_k s_i.
		\end{align*}
		If $k \ge {i+2}$, we have \[s_i t_k\stackrel{(f)}{=}s_i s_k \cdots s_{n-1}t_n s_{n-1} \cdots s_k \stackrel{(a) , (d)}{=}s_k \cdots s_{n-1}t_n s_{n-1} \cdots s_k s_i \stackrel{(f)}{=}t_k s_i.\]
		 \ref{(iv)} If $i\le{n-1}$, it follows
		\begin{align*}
			t_i t_n&\stackrel{(f)}{=}s_i \cdots s_{n-1}t_n s_{n-1} \cdots s_i t_n \\
			&\stackrel{(d)}{=}s_i \cdots s_{n-1}t_n s_{n-1} t_n \cdots s_i \\
			&\stackrel{(c)}{=}s_i \cdots t_{n}s_{n-1} t_{n} s_{n-1} \cdots s_i \\
			&\stackrel{(d)}{=}t_n s_i \cdots s_{n-1} t_{n} s_{n-1} \cdots s_i \stackrel{(f)}{=}t_n t_i.
		\end{align*}
		If $i=j$, it follows $t_i t_j=t_j t_i$. Without loss of generality, assume $i<j$. Therefore,
		\begin{align*}
			t_i t_j&\stackrel{(f)}{=}t_i s_j \cdots s_{n-1}t_n s_{n-1} \cdots s_j \\
			&\stackrel{(d)}{=}s_j \cdots s_{n-1} t_i t_n s_{n-1} \cdots s_j \\
			&\stackrel{(iv)}{=}s_j \cdots s_{n-1} t_n t_i s_{n-1} \cdots s_j \\
			&\stackrel{(d)}{=}s_j \cdots s_{n-1} t_n s_{n-1} \cdots s_j t_i\stackrel{(f)}{=}t_j t_i.
		\end{align*}
		Here the second and fourth equations follows from (iii).
	\end{proof}
	Let $rep(G(r,1,n))$ denote the set of all finite dimensional complex linear representations of $G(r,1,n)$. For any $\rho \in rep(G(r,1,n))$, the \textit{characteristic polynomial of} $\rho$ is defined as \[d(S,\rho)=det(x_0I+x_1\rho(s_1)+\cdots+x_{n-1}\rho(s_{n-1})+x_n\rho(t_1)+\cdots+x_{2n-1}\rho(t_n)).\]
	In this section, we will use \textit{words} to denote the expressions of elements in a given finitely generated group $G(r,1,n)$ with $S=\{s_1,\cdots,s_{n-1},t_1,\cdots,t_n\}$ as its generating set, which means elements of the free monoid $M(S)$ generated by alphabet $S$. Here we define some notations related to words. The length of a word $\omega \in M(S)$ is the number of letters in $\omega$, denoted by $\lvert \omega \rvert$. The identity of $M(S)$ is denoted by $1$. There is a natural homomorphism \[\phi:M(S)\longrightarrow G(r,1,n),\]
	where $\phi(s_i)=s_i,\phi(t_i)=t_i$. We often just use $\omega$ to  express $\phi(\omega)$ when consider it as an element in $G(r,1,n)$ with no confusion. But note that when we say $\omega_1=\omega_2$, it always means $\omega_1$ is equal to $\omega_2$ as elements in monoid $M(S)$. If we want to express that $\omega_1$ is equal to $\omega_2$ as elements in $G(r,1,n)$, we say $\phi(\omega_1)=\phi(\omega_2)$.
	
	\begin{defn} Let $G(r,1,n)$ be the imprimitive group with $S=\{s_1,\cdots,$ $s_{n-1},t_1,\cdots,t_n\}$ as in Definition 2.1. For a word $\omega \in M(S)$, the \textit{signature} of $\omega$ is $sig(\omega)=(a_1,b_1,a_2,b_2,\cdots,a_{n-1},$ $b_{n-1},b_n)\in \mathbb{N}^{2n-1}$ where $a_i$ and $b_i$ are the number of times $s_i$ and $t_i$ appearing in $\omega$, respectively. For each signature \[\alpha=(a_1,b_1,a_2,b_2,\cdots,a_{n-1},b_{n-1},b_n),\] the length of $\alpha$ is \[\lvert \alpha \rvert=\sum_{i=1}^{n-1} a_i +\sum_{i=1}^n b_i.\]
	\end{defn}
	
	It is well known that if two representations $\rho_1,\rho_2$ of $G$ satisfy $\chi_{\rho_1}=\chi_{\rho_2}$, where $\chi_{\rho_i}$ is the character of $\rho_i$, then $\rho_1 \sim \rho_2$.
	Similar to \cite[Theorem 1.1]{CS2021} and \cite[Theorem 3.2]{PS2024}, we have the following theorem.
	More facts about linear representations of groups could be found in \cite{S1977}.

	\begin{thm} Let $G(r,1,n)$ be the imprimitive group with generating set $S=\{s_1,\cdots,s_{n-1},t_1,$ $\cdots,t_n\}$. Let $\rho_1,\rho_2 \in rep(G(r,1,n))$. If $d(S,\rho_1)=d(S,\rho_2)$, then for each given $\alpha \in \mathbb{N}^{2n-1}$, we have \[\sum_{sig(\omega)=\alpha} \chi_{\rho_1}(\omega)=\sum_{sig(\omega)=\alpha} \chi_{\rho_2}(\omega),\]
		\textit{where $\chi_{\rho_i}$ denotes the character of representation $\rho_i$, for $i=1,2$.}
	\end{thm}
	\begin{proof} If $d(S,\rho_1)=d(S,\rho_2)$, then for each $(x_1,\cdots,x_{2n-1})\in \mathbb{C}^{2n-1}$ we have
		%\begin{multline*}
		\begin{eqnarray*}
			\det(\lambda I-(\sum_{i=1}^{n-1} x_i\rho_1(s_i)+\sum_{i=1}^n x_{n-1+i}\rho_1(t_i)))\\
			=\det(\lambda I-(\sum_{i=1}^{n-1} x_i\rho_2(s_i)+\sum_{i=1}^n x_{n-1+i}\rho_2(t_i))).
		\end{eqnarray*}
		%\end{multline*}
		By using Spectral mapping theorem(refer to \cite{C1990}) of matrices, for each $m\in \mathbb{N}_+$ we have
		\begin{eqnarray*}
			\det(\lambda I-(\sum_{i=1}^{n-1} x_i\rho_1(s_i)+\sum_{i=1}^n x_{n-1+i}\rho_1(t_i))^m)\\
			=\det(\lambda I-(\sum_{i=1}^{n-1} x_i\rho_2(s_i)+\sum_{i=1}^n x_{n-1+i}\rho_2(t_i))^m).
		\end{eqnarray*}
		Hence \[Tr((\sum_{i=1}^{n-1} x_i\rho_1(s_i)+\sum_{i=1}^n x_{n-1+i}\rho_1(t_i))^m)=Tr((\sum_{i=1}^{n-1} x_i\rho_2(s_i)+\sum_{i=1}^n x_{n-1+i}\rho_2(t_i))^m).\]
		
		When we evaluate $(x_1,\cdots,x_{2n-1})$ with complex numbers, the two sides of equation are equal as polynomials, which implies that the corresponding coefficients of two sides are equal. Considering the value of $m$, we have \[\sum_{sig(\omega)=\alpha} \chi_{\rho_1}(\omega)=\sum_{sig(\omega)=\alpha} \chi_{\rho_2}(\omega),\]
		for each signature $\alpha\in \mathbb{N}^{2n-1}$.
	\end{proof}
	
	\begin{defn} Let $u$ be a word of $G(r,1,n)$ and $sig(u)=(a_1,b_1,\cdots,a_{n-1},$ $b_{n-1},b_n)$. Define \[ct(u)\triangleq(\lvert u \rvert,a_1,b_1,\cdots,a_{n-1},b_{n-1}).\]
		Define an order on the set $M(S)$ of all the words of $G(r,1,n)$ as \[u_1\le u_2 \quad \text{if} \quad ct(u_1) \le_{lex} ct(u_2), \]
		where $\le_{lex}$ is the lexicographic order on $\mathbb{N}^{2n-1}$. Clearly $(M(S),\le)$ has a minimum element $1$ with a length of 0.
	\end{defn}
	
	Inspired by \cite{BB2005,CS2021}, we define the echelon form.
	
	\begin{defn} We say a word is in \textit{echelon form} if it is of the form \[\omega = \delta_1\delta_2\cdots\delta_n,\]
		where for each $i$ the word $\delta_i$ satisfies
		\begin{equation}\label{2.1}
			\delta_i = \left\{\begin{aligned}
				&1,s_{i},t_i,\cdots,t_i^{r-2} \, or \, t_i^{r-1}\quad &\text{if}\,\, i<n, \\
				&1,t_i,\cdots,t_i^{r-2} \, or \, t_i^{r-1}      \quad &\text{if}\,\, i=n.
			\end{aligned}
			\right.
		\end{equation}
	\end{defn}
	\begin{rem} Note that when $\omega$ is in echelon form, it can be recovered from $ct(\omega)$. More precisely, given $ct(\omega)=(\lvert \omega \rvert ,a_1,b_1,\cdots,a_{n-1},b_{n-1})$, the tuple $(\lvert \omega \rvert ,a_1,b_1,\cdots,a_{n-1},b_{n-1})$ determines the tuple $(\delta_1,\cdots,\delta_n)$.
	\end{rem}
	To make any word conjugate to a word in echelon form, we need to introduce admissible transformations. This definition plays a crucial role in the subsequent work.
	\begin{defn} We introduce \textit{admissible transformation} on words of $M(S)$.
		
		(1)\textit{Canceling transformations}. That is, for each $i$ the transformation
		\begin{equation}\label{2.2}
			\omega^\prime s_i s_i \omega^{\prime \prime} \mapsto \omega^\prime \omega^{\prime \prime},\  1\le i \le n-1
		\end{equation}
		\begin{equation}\label{2.3}
			\omega^\prime \underbrace{t_i \cdots t_i}_r \omega^{\prime \prime} \mapsto \omega^\prime \omega^{\prime \prime},\ 1 \le i \le n
		\end{equation}
		are admissible.
		
		(2)\textit{Commuting transformations}. The transformations are below.
		\begin{equation}\label{2.4}\omega^\prime s_i s_j \omega^{\prime \prime} \mapsto \omega^\prime s_j s_i \omega^{\prime \prime},\ \lvert i-j \rvert >1
		\end{equation}
		\begin{equation}\label{2.5}
			\omega^\prime t_i t_j \omega^{\prime \prime} \mapsto \omega^\prime t_j t_i \omega^{\prime \prime},\ 1 \le i,j \le n
		\end{equation}
		\[\omega^\prime t_k s_i \omega^{\prime \prime} \mapsto \omega^\prime s_i t_k \omega^{\prime \prime} \quad\text{and}\]
		\begin{equation}\label{2.6}
			\omega^\prime s_i t_k \omega^{\prime \prime} \mapsto \omega^\prime t_k s_i \omega^{\prime \prime}, \  k \le {i-1}\  \text{or} \  k\ge {i+2}
		\end{equation}
		
		(3)\textit{Replacement transformations}. The transformations are below.
		\begin{equation}\label{2.7}\omega^\prime s_i s_{i+1} s_i \omega^{\prime \prime} \mapsto \omega^\prime s_{i+1} s_i s_{i+1}\omega^{\prime \prime},\ 1 \le i\le {n-2}
		\end{equation}
		\[\omega^\prime s_i t_i \omega^{\prime \prime} \mapsto \omega^\prime t_{i+1} s_i \omega^{\prime \prime},\ 1\le i \le {n-1} \quad\text{and}\]
		\begin{equation}\label{2.8}\omega^\prime t_i s_i \omega^{\prime \prime} \mapsto \omega^\prime s_{i} t_{i+1} \omega^{\prime \prime},\ 1\le i \le {n-1}
		\end{equation}
		\begin{equation}\label{2.9}\omega^\prime s_i t_{i+1}^k s_i \omega^{\prime \prime} \mapsto \omega^\prime t_{i}^k \omega^{\prime \prime},\ 1\le i \le {n-1}
		\end{equation}
		
		(4)\textit{Circular transformations}. Those are
		\begin{equation}\label{2.10}\delta_{i_1} \cdots \delta_{i_k}\delta_{i_{k+1}} \cdots \delta_{i_N} \mapsto \delta_{i_{k+1}} \cdots \delta_{i_N}\delta_{i_{1}} \cdots \delta_{i_k},
		\end{equation}
		where $\delta_{i_j} \in \{s_1,\cdots,s_{n-1},t_1,\cdots,t_n \}$.
	\end{defn}
	We write $\omega_1 \leadsto \omega_2$ if there is a sequence of admissible transformations that map $\omega_1$ into $\omega_2$.
	
	\begin{lemma} If $\omega_1 \leadsto \omega_2$, then $\omega_1$ is conjugate to $\omega_2$ and $ct(\omega_1) \ge_{lex} ct(\omega_2)$.
	\end{lemma}
	\begin{proof} It can be verified case by case for the admissible transformations in the above.
	\end{proof}

	\section{Characteristic polynomials for  $G(r,1,n)$ and $G(r_1,1,n_1)\times G(r_2,1,n_2)$}
	\subsection{Cases for $G(r,1,n)$ }
	\hspace{1.3em} In this section, we will prove that every word $\omega$ can be transformed into echelon form through admissible transformations. Before proving it, we need some lemmas.
	\begin{lemma}Let $G(r,1,2)$ be the imprimitive group with $S=\{s_1,t_1,t_2\}$ as in Definition 2.1. For a word $\omega \in M(S)$, let $sig(\omega)=(a_1,b_1,b_2)\in \mathbb{N}^{3}$ be the \textit{signature} of $\omega$. If $i$ is the smallest index of a letter in $\omega$, then $\omega \leadsto \omega^\prime \delta_i \omega^{\prime \prime}$, where the words $\omega^\prime$ and $\omega^{\prime \prime}$ only contain letters with index greater than $i$, and $\delta_i$ satisfies (\ref{2.1}). Furthermore, the sequence of admissible transformations can be chosen so that no circular transformations are used.
	\end{lemma}
	\begin{proof}The case when $i= 2$ is trivial, so we assume that $i=1$. If $a_1+b_1=1$ then no transformations are needed, so we may assume that $a_1 + b_1 \ge 2$. Thus we have $\omega= u \delta v \delta^\prime s$ where $u$ and $v$ do not involve $s_1$ and $t_1$, and $\delta,\ \delta^\prime =s_1 \ \text{or} \  t_1$.
		
		Now, we prove this lemma by induction on $a_1+b_1$. Suppose first that $a_1+b_1=2$. Then $s$ does not involve $s_1$ and $t_1$. And $\omega= us_1vt_1s,ut_1vs_1s,us_1vs_1s$ $\text{or}\  ut_1vt_1s$, where $u$, $v$ and $s$ do not involve $s_1$ and $t_1$, hence they are all powers of $t_2$. Since $v$ can only involve $t_2$, by a sequence of commuting, replacement and canceling transformations, it yields that
		\begin{flushleft}
			\begin{enumerate}
				\item $\omega=us_1vt_1s \stackrel{(\ref{2.5})}{\leadsto} us_1t_1vs \stackrel{(\ref{2.8})}{\leadsto} ut_2s_1vs=ut_2 \delta_1 vs, \, \delta_1=s_1$,
				\item $\omega= ut_1vs_1s \stackrel{(\ref{2.5})}{\leadsto} uvt_1s_1s \stackrel{(\ref{2.8})}{\leadsto} uvs_1t_2s=uv \delta_1 t_2 s,\ \delta_1=s_1$,
				\item $\omega= us_1vs_1s = us_1t_2^ks_1s \stackrel{(\ref{2.9})}{\leadsto} ut_1^ks \stackrel{(\ref{2.3})}{\leadsto} ut_1^{k^\prime}s=u \delta_1  s,\ \delta_1=t_1^{k^\prime},\, 0\le k^{\prime} \le {r-1}$,
				\item $\omega=ut_1vt_1s \stackrel{(\ref{2.5})}{\leadsto} uvt_1^2 s=uv\delta_1s,\, \delta_1=t_1^2$.
			\end{enumerate}
		\end{flushleft}
		Suppose next that $a_1 +b_1 \ge 3$. Then $s=s^\prime \delta^{\prime \prime} s^{\prime \prime}$ where $s^{\prime \prime}$ does not involve $s_1$ and $t_1$, and $\delta^{\prime \prime}=s_1 \,\, \text{or} \,\, t_1$. Now, $\omega=u\delta v\delta^{\prime}s^\prime \delta^{\prime \prime} s^{\prime \prime}$, and by induction we have $u\delta v\delta^{\prime}s^\prime \leadsto u^{\prime} \delta_1 v^{\prime}$ where $u^{\prime}$ and $v^{\prime}$ do not involve $s_1$ and $t_1$ and no circular transformations were used, hence the same sequence of transformations yields $\omega \leadsto u^{\prime}\delta_1 v^{\prime}\delta^{\prime \prime} s^{\prime \prime}$. Because $\delta_1 \in \{1,s_1,t_1,\cdots,t_1^{r-1}\}$ and $\delta^{\prime \prime} \in \{s_1,t_1\}$ when $\delta_1 =1$, it is trivial, so only the four cases $u^{\prime}s_1 v^{\prime} s_1 s^{\prime \prime},u^{\prime}s_1 v^{\prime} t_1 s^{\prime \prime},u^{\prime}t_1^k v^{\prime} s_1 s^{\prime \prime}$ and $u^{\prime}t_1^k v^{\prime} t_1 s^{\prime \prime}$ need to be considered. Since $v^{\prime}$ can only involve $t_2$, a sequence of commuting, replacement and canceling transformations yields that
		\begin{flushleft}
			\begin{enumerate}
				\item $u^{\prime}s_1 v^{\prime} s_1 s^{\prime \prime} \stackrel{(\ref{2.3})}{\leadsto} u^{\prime}s_1 t_2^{k^\prime} s_1 s^{\prime \prime} \stackrel{(\ref{2.9})}{\leadsto} u^{\prime} t_1^{k^\prime} s^{\prime \prime}=u^\prime \delta_1 s^{\prime \prime}, \, \delta_1=t_1^{k^\prime} , \, 0\le k^\prime \le {r-1}$,
				\item $u^{\prime}s_1 v^{\prime} t_1 s^{\prime \prime} \stackrel{(\ref{2.5})}{\leadsto} u^{\prime}s_1 t_1 v^{\prime} s^{\prime \prime} \stackrel{(\ref{2.8})}{\leadsto} u^{\prime}t_2 s_1 v^{\prime} s^{\prime \prime}=u^{\prime}t_2 \delta_1 v^{\prime} s^{\prime \prime}, \, \delta_1=s_1$,
				\item $u^{\prime}t_1^k v^{\prime} s_1 s^{\prime \prime}  \stackrel{(\ref{2.5})}{\leadsto} u^{\prime} v^{\prime} t_1^k s_1 s^{\prime \prime} \stackrel{(\ref{2.8})}{\leadsto} u^{\prime} v^{\prime} s_1 t_2^k s^{\prime \prime}=u^{\prime} v^{\prime} \delta_1 t_2^k s^{\prime \prime}, \, \delta_1=s_1,$
				\item $u^{\prime}t_1^k v^{\prime} t_1 s^{\prime \prime}, \stackrel{(\ref{2.5})}{\leadsto} u^{\prime} v^{\prime} t_1^{k+1} s^{\prime \prime} \stackrel{(\ref{2.3})}{\leadsto} u^{\prime} v^{\prime} t_1^{l^\prime} s^{\prime \prime}=u^{\prime} v^{\prime} \delta_1 s^{\prime \prime},\, \delta_1=t_1^{l^\prime}, \, 0 \le l^{\prime} \le {r-1}$.
			\end{enumerate}
		\end{flushleft}
		This completes the proof for the lemma.
	\end{proof}
	The case $n=2$ has the conclusions of the above lemma, and it is natural to consider whether the same conclusion holds for $n=3$. Next, we consider the case of $G(r,1,3)$.
	\begin{lemma}Let $G(r,1,3)$ be the imprimitive group with $S=\{s_1,t_1,s_2,t_2,t_3\}$ as in Definition 2.1. For a word $\omega \in M(S)$, let $i$ be the smallest index of a letter in $\omega$, then $\omega \leadsto \omega^\prime \delta_i \omega^{\prime \prime}$, where the words $\omega^\prime$ and $\omega^{\prime \prime}$ only contain letters with index greater than $i$, and $\delta_i$ satisfies (\ref{2.1}). Furthermore, the sequence of admissible transformations can be chosen so that no circular transformations are used.
	\end{lemma}
	\begin{proof}The case when $i= 3$ is trivial, so we may assume that $i\le {2}$, hence without loss of generality also that $i=1$. If $a_1+b_1=1$ then no transformations are needed, so we may assume that $a_1 + b_1 \ge 2$. Thus we have $\omega= u \delta v \delta^\prime s$ where $u$ and $v$ do not involve $s_1$ and $t_1$, and $\delta,\delta^\prime =s_1 \ \text{or} \  t_1$.
		
		Let $G(r,1,2)$ be the imprimitive group with $S^\prime=\{s_2,t_2,t_3\}$ as in Definition 2.1. By Lemma 3.1, we have $v \leadsto v^\prime \delta_2 v^{\prime \prime}$ where $v^\prime$ and $v^{\prime \prime}$ can only involve $t_3$, and no circular transformations are used. Therefore the same sequence of admissible transformations yields $\omega \leadsto u\delta v^\prime \delta_2 v^{\prime \prime} \delta^\prime s$. Since $\delta$ commutes with $t_3$, a sequence of commuting transformations yields $\omega \leadsto u v^\prime \delta \delta_2 v^{\prime \prime} \delta^\prime s$. Similarly, $u v^\prime \delta \delta_2 v^{\prime \prime} \delta^\prime s \leadsto u v^\prime \delta \delta_2  \delta^\prime v^{\prime \prime} s$. Now, we prove this lemma by induction on $a_1+b_1$.
		
		  \begin{Case}
       	 \item When $a_1+b_1=2$, we have $s$ does not involve $s_1$ and $t_1$.

         \item When $\delta_2=1$, $u v^\prime \delta \delta_2  \delta^\prime v^{\prime \prime} s=u v^\prime \delta \delta^\prime v^{\prime \prime} s$, there are the following four subcases:
		\begin{Case}
			%\begin{enumerate}
				\item $uv^\prime s_1 s_1 v^{\prime \prime}s \stackrel{(\ref{2.2})}{\leadsto} uv^\prime 1 v^{\prime \prime}s $.
				\item $uv^\prime t_1 s_1 v^{\prime \prime}s \stackrel{(\ref{2.8})}{\leadsto} uv^\prime s_1 t_2 v^{\prime \prime}s $.
				\item $uv^\prime s_1 t_1 v^{\prime \prime}s  \stackrel{(\ref{2.8})}{\leadsto} uv^\prime t_2 s_1 v^{\prime \prime}s $.
				\item $uv^\prime t_1 t_1 v^{\prime \prime}s = uv^\prime t_1^2 v^{\prime \prime}s $.
			%\end{enumerate}
		\end{Case}
	\item	When $\delta_2=s_2$, $u v^\prime \delta \delta_2  \delta^\prime v^{\prime \prime} s=u v^\prime \delta s_2\delta^\prime v^{\prime \prime} s$, there are the following four subcases:
		\begin{Case}
			%\begin{enumerate}
				\item $uv^\prime s_1 s_2 s_1 v^{\prime \prime}s \stackrel{(\ref{2.7})}{\leadsto} uv^\prime s_2 s_1 s_2 v^{\prime \prime}s $.
				\item $uv^\prime t_1 s_2 s_1 v^{\prime \prime}s \stackrel{(\ref{2.6})}{\leadsto} uv^\prime s_2 t_1 s_1 v^{\prime \prime}s \stackrel{(\ref{2.8})}{\leadsto} u v^\prime s_2 s_1 t_2 v^{\prime \prime} s$.
				\item $uv^\prime s_1 s_2 t_1 v^{\prime \prime}s  \stackrel{(\ref{2.6})}{\leadsto} uv^\prime s_1 t_1 s_2 v^{\prime \prime}s \stackrel{(\ref{2.8})}{\leadsto} uv^\prime t_2 s_1 s_2 v^{\prime \prime}s$.
				\item $uv^\prime t_1 s_2 t_1 v^{\prime \prime}s \stackrel{(\ref{2.6})}{\leadsto} uv^\prime s_2 t_1^2 v^{\prime \prime}s.$
			%\end{enumerate}
		\end{Case}
	\item	When $\delta_2=t_2^k$, $u v^\prime \delta \delta_2  \delta^\prime v^{\prime \prime} s=u v^\prime \delta t_2^k \delta^\prime v^{\prime \prime} s$, there are the following four subcases:
		\begin{Case}
			%\begin{enumerate}
				\item $uv^\prime s_1 t_2^k s_1 v^{\prime \prime}s \stackrel{(\ref{2.3})}{\leadsto} uv^\prime s_1 t_2^{k_1} s_1 v^{\prime \prime}s \stackrel{(\ref{2.9})}{\leadsto} uv^\prime t_1^{k_1} v^{\prime \prime}s,\, 0 \le k_1 \le {r-1}$.
				\item $uv^\prime t_1 t_2^k s_1 v^{\prime \prime}s \stackrel{(\ref{2.5})}{\leadsto} uv^\prime t_2^k t_1 s_1 v^{\prime \prime}s \stackrel{(\ref{2.8})}{\leadsto} uv^\prime t_2^k s_1 t_2 v^{\prime \prime}s$.
				\item $uv^\prime s_1 t_2^k t_1 v^{\prime \prime}s \stackrel{(\ref{2.5})}{\leadsto} uv^\prime s_1 t_1 t_2^k v^{\prime \prime}s \stackrel{(\ref{2.8})}{\leadsto} uv^\prime t_2 s_1 t_2^k v^{\prime \prime}s$.
				\item $uv^\prime t_1 t_2^k t_1 v^{\prime \prime}s \stackrel{(\ref{2.5})}{\leadsto} uv^\prime t_2^k t_1^2 v^{\prime \prime}s$.
			%\end{enumerate}
		\end{Case}
		
			\end{Case}

		All three cases above are in the desired form. Now we consider $a_1+b_1 \ge 3$.
        We can still observe that $\omega \leadsto u v^\prime \delta \delta_2  \delta^\prime v^{\prime \prime} s$ (Note that the $s$ here involves $s_1$ or $t_1$).
        From the results above, it is known that when $\delta_2=1$, subcases 2.1, 2.2 and 2.3 have smaller values of $a_1+b_1$. By induction on $a_1+b_1$, we immediately deduce the desired results. In subcase 2.4, induction on $a_1+b_1$ implies $s \leadsto z_1 \delta_1 z_1^\prime $ where $z_1$ and $z_1^\prime$ do not involve $s_1$ and $t_1$, and no circular transformations are uesd. Therefore the same sequence of transformations yields $uv^\prime t_1^2 v^{\prime \prime}s \leadsto uv^\prime t_1^2 v^{\prime \prime}z_1 \delta_1 z_1^\prime$.
        If $\delta_1=1$, it follows that
        \[uv^\prime t_1^2 v^{\prime \prime}z_1 \delta_1 z_1^\prime=uv^\prime t_1^2 v^{\prime \prime}z_1  z_1^\prime\] is in the desired form.

        If $\delta_1=s_1$,
        \[uv^\prime t_1^2 v^{\prime \prime}z_1 \delta_1 z_1^\prime=uv^\prime t_1^2 v^{\prime \prime}z_1 s_1 z_1^\prime \leadsto uv^\prime v^{\prime \prime}z_1 t_1^2 s_1 z_1^\prime \leadsto uv^\prime v^{\prime \prime}z_1 s_1 t_2^2 z_1^\prime\] is in the desired form.

        If $\delta_1=t_1^k$,
        \[uv^\prime t_1^2 v^{\prime \prime}z_1 \delta_1 z_1^\prime=uv^\prime t_1^2 v^{\prime \prime}z_1 t_1^k z_1^\prime \leadsto uv^\prime v^{\prime \prime}z_1 t_1^{k+2} z_1^\prime \leadsto uv^\prime v^{\prime \prime}z_1 t_1^{k_0} z_1^\prime\] is in the desired form, where $0 \le k_0 \le {r-1}$.

        In a similar manner, the cases of $\delta_2=s_2$ and $\delta_2=t_2^k$ can also yield the desired results.
	\end{proof}
	Finally, we consider the case of $G(r,1,n)$, and prove that it has the conclusions of the above.
	\begin{lemma} Let $i$ be the smallest index of a letter in $\omega$. Then $\omega \leadsto \omega^\prime \delta_i \omega^{\prime \prime}$, where the words $\omega^\prime$ and $\omega^{\prime \prime}$ can only contain letters with index greater than $i$, and $\delta_i$ satisfies (\ref{2.1}). Furthermore, the sequence of admissible transformations can be chosen so that no circular transformations are used.
	\end{lemma}
	\begin{proof} The case when $i= n$ is trivial, so we may assume that $i\le {n-1}$. Without loss of generality, assume $i=1$. If $a_1+b_1=1$ then no transformations are needed, so we may assume that $a_1 + b_1 \ge 2$. Thus we have $\omega= u \delta v \delta^\prime s$ where $u$ and $v$ do not involve $s_1$ and $t_1$, and $\delta,\delta^\prime =s_1 \ \text{or} \  t_1$.
		
		The case $n=2$ has already been proven by Lemma 3.1. Next, suppose $n \ge 3$. If $a_1+b_1=2$, then s does not involve $s_1$ and $t_1$. By induction on $n$ we have that $v \leadsto v^\prime \delta_2 v^{\prime \prime}$ where $v^\prime$ and $v^{\prime \prime}$ can only involve letters with index $3$ or higher, and no circular transformations are used. Therefore the same sequence of admissible transformations yields $\omega \leadsto u\delta v^\prime \delta_2 v^{\prime \prime} \delta^\prime s$. Since $\delta$ commutes with all the letters appearing in $v^\prime$, a sequence of commuting transformations yields $\omega \leadsto u v^\prime \delta \delta_2 v^{\prime \prime} \delta^\prime s$. Similarly, we have $u v^\prime \delta \delta_2 v^{\prime \prime} \delta^\prime s \leadsto u v^\prime \delta \delta_2  \delta^\prime v^{\prime \prime} s$. Let $G(r,1,3)$ be the imprimitive group with $S^\prime=\{s_1,t_1,s_2,t_2,t_3\}$ as in Definition 2.1. Then $\delta \delta_2  \delta^\prime \in M(S^\prime)$. By Lemma 3.2, we have $\delta \delta_2  \delta^\prime \leadsto \omega^\prime \delta_1 \omega^{\prime \prime}$ where $\omega^\prime$ and $\omega^{\prime \prime}$ do not involve $s_1$ and $t_1$. Therefore the same sequence of transformations yields $u v^\prime \delta \delta_2  \delta^\prime v^{\prime \prime} s \leadsto u v^\prime \omega^\prime \delta_1 \omega^{\prime \prime} v^{\prime \prime} s$.
		
		When $a_1+b_1 \ge 3$, $s=s^\prime \delta^{\prime \prime} s^{\prime \prime}$ where $s^{\prime \prime}$ does not involve $s_1$ and $t_1$, and $\delta^{\prime \prime}=s_1 \  \text{or} \  t_1$. Now, $\omega=u\delta v\delta^{\prime}s^\prime \delta^{\prime \prime} s^{\prime \prime}$, and by induction we have $u\delta v\delta^{\prime}s^\prime \leadsto u^{\prime} \delta_1 v^{\prime}$ where $u^{\prime}$ and $v^{\prime}$ do not involve $s_1$ and $t_1$ and no circular transformations were used, hence the same sequence of transformations yields $\omega \leadsto u^{\prime}\delta_1 v^{\prime}\delta^{\prime \prime} s^{\prime \prime}$. By induction on $n$, we have that $v^\prime \leadsto v_1 \delta_2 v_2$ where $v_1$ and $v_2$ can only involve letters with index $3$ or higher, and no circular transformations are used. Therefore the same sequence of admissible transformations yields $\omega \leadsto u^\prime \delta_1 v_1 \delta_2 v_2 \delta^{\prime \prime} s^{\prime \prime}$.
		Since $\delta_1$ commutes with all the letters appearing in $v_1$, a sequence of commuting transformations yields $\omega \leadsto u^\prime v_1 \delta_1 \delta_2 v_2 \delta^{\prime \prime} s^{\prime \prime}$. Similarly, $u^\prime v_1 \delta_1 \delta_2 v_2 \delta^{\prime \prime} s^{\prime \prime} \leadsto u^\prime v_1 \delta_1 \delta_2 \delta^{\prime \prime} v_2 s^{\prime \prime}$. By Lemma 3.2, we have $\delta_1 \delta_2 \delta^{\prime \prime} \leadsto \omega^\prime \delta_1 \omega^{\prime \prime}$ where $\omega^\prime$ and $\omega^{\prime \prime}$ do not involve $s_1$ and $t_1$. Therefore the same sequence of transformations yields that $u^\prime v_1 \delta_1 \delta_2 \delta^{\prime \prime} v_2 s^{\prime \prime} \leadsto u^\prime v_1 \omega^\prime \delta_1 \omega^{\prime \prime} v_2 s^{\prime \prime}$. This completes the proof of the lemma.
	\end{proof}
	Now, we will utillize the aforementioned lemmas to prove that the word $\omega$ can be transformed into echelon form through admissible transformations.
	\begin{thm} Let $\omega$ be a word. Then there exists a word $\tilde{\omega}$ in echelon form such that $\omega \leadsto \tilde{\omega}$, in particular $\omega$ is conjugate to $\tilde{\omega}$ and $ct(\omega) \ge_{lex} ct(\tilde{\omega})$.
	\end{thm}
	\begin{proof} Let $i_1$ be the smallest index of a letter in $\omega$. By Lemma 3.3 we get $\omega \leadsto \omega^\prime \delta_{i_1} \omega^{\prime \prime}$, hence a circular transformation produces $\omega \leadsto \delta_{i_1} \omega^{\prime \prime}\omega^\prime=\delta_1 \cdots \delta_{i_1}\omega_{i_1}$ where $\omega_{i_1}$ only involves letters with index greater than $i_1$. Suppose for some $k \le {n-1}$ we have already obtained $\omega \leadsto \delta_1 \cdots \delta_k \omega_k$ where $\omega_k$ can only involve letters with index greater than $k$. If $k< {n-1}$, by Lemma 3.3, we have $\omega_k \leadsto \omega_k^\prime \delta_{k+1} \omega_k^{\prime \prime}$ where no circular transformations are used, and $\omega_k$ and $\omega_k^{\prime \prime}$ can only involve letters with index greater than $k+1$. Thus the same sequence of transformations yields $\omega \leadsto \delta_1 \cdots \delta_k \omega_k^\prime \delta_{k+1} \omega_k^{\prime \prime}$. Therefore a sequence of commuting transformations produces $\omega \leadsto \omega_k^\prime \delta_1 \cdots \delta_k \delta_{k+1} \omega_k^{\prime \prime}$. Now a circular transformation gives $\omega \leadsto \delta_1 \cdots \delta_{k+1} \omega_{k+1}$ where $\omega_{k+1}= \omega_k^{\prime \prime}\omega_k^\prime$ can only involve letters of index greater than $k+1$. Repeat this process until $k=n-1$. Then $\omega \leadsto \delta_1 \cdots \delta_{n-1} \omega_{n-1}$ where $\omega_{n-1}$ involves only the letter $t_n$. Therefore, the theorem holds.
	\end{proof}
	\begin{lemma} Let $\omega_1,\omega_2$ be words in G(r,1,n). If there exists words $\tilde{\omega_1}$ and $\tilde{\omega_2}$ in echelon form such that $\omega_1 \leadsto \tilde{\omega_1},\omega_2 \leadsto \tilde{\omega_2}$ and $ct(\tilde{\omega_1})=ct(\tilde{\omega_2})$, it follows that $\omega_1$ is conjugate to $\omega_2$.
	\end{lemma}
	\begin{proof} By Lemma 2.9, we obtain that $\omega_1$ and $\omega_2$ are conjugate to $\tilde{\omega_1}$ and $\omega_2$, respectively. By Remark 2.7, it follows that $\tilde{\omega_1}=\tilde{\omega_2}$, which implies that $\omega_1$ is conjugate to $\omega_2$.
	\end{proof}
	\begin{thm} Let $G(r,1,n)$ be a imprimitive group with $S=\{s_1,\cdots,s_{n-1},$ $t_1,\cdots,t_n\}$ as its generating set. For two finite dimensional complex linear representations $\rho_1$ and $\rho_2$ of $G(r,1,n)$, the characteristic polynomials $d(S,\rho_1)=d(S,\rho_2)$ if and only if $\rho_1\cong \rho_2$.
	\end{thm}
		\begin{proof} If $\rho_1 \cong \rho_2$, there exists an invertible matrix $X$ such that $\rho_2(\omega)=X^{-1}\rho_1(\omega)X$ for each $\omega \in G(r,1,n)$. It follows that
    \begin{eqnarray*}
    d(S,\rho_2)&=&det(x_0I+\sum_{i=1}^{n-1} x_i\rho_2(s_i)+\sum_{i=1}^n x_{n-1+i}\rho_2(t_i)) \\
    &=&det(X^{-1}(x_0I+\sum_{i=1}^{n-1} x_i\rho_1(s_i)+\sum_{i=1}^n x_{n-1+i}\rho_1(t_i))X)\\
    &=&det((x_0I+\sum_{i=1}^{n-1} x_i\rho_1(s_i)+\sum_{i=1}^n x_{n-1+i}\rho_1(t_i)))\\
    &=&d(S,\rho_1).
 \end{eqnarray*}
  If $d(S,\rho_1)=d(S,\rho_2)$, then by Theorem 2.4 we have
		\begin{equation}
			\sum_{sig(\omega)=\alpha} \chi_{\rho_1}(\omega)=\sum_{sig(\omega)=\alpha} \chi_{\rho_2}(\omega),
		\end{equation}
		for each signature $\alpha\in \mathbb{N}^{2n-1}$.
		
		Recall that the character of a finite dimensional representation $\rho$ is the class function $\chi_\rho (\omega)=Tr(\rho(\omega))$. Since the character determines the representation up to an equivalence, in the remaining cases for $G(r,1,n)$, it suffices to show by induction on the partial ordering of words that for each word $\omega$ in $G(r,1,n)$ we have
		\[Tr(\rho_1(\omega))=Tr(\rho_2(\omega)).\]
		When $\omega=1$, $sig(\omega)=(0,\cdots,0) \in \mathbb{N}^{2n-1}$, we have
		\[Tr(\rho_1(1))=\sum_{sig(u)=(0,\cdots,0)} \chi_{\rho_1}(u)=\sum_{sig(u)=(0,\cdots,0)} \chi_{\rho_2}(u)=Tr(\rho_2(1)).\]
		Suppose $\omega\neq1$ and we have proved our statement for all words smaller than $\omega$ for the partial ordering of words. Let $sig(\omega)=\alpha=(a_1,b_1,\cdots,a_{n-1},b_{n-1},b_n)$, for each word $u$ with $sig(u)=\alpha$ let $\tilde{u}$ be a word in echelon form such that $u \leadsto \tilde{u}$, as proved in the Theorem 3.4. Thus, by Lemma 2.9, we have $ct(\omega)=ct(u) \ge_{lex} ct(\tilde{u})$. Let
		\[c=\lvert \{u\ |sig(u)=\alpha \, \text{and}\,  ct(\tilde{u})=ct(\omega)\}\rvert .\]
		Since $u$ and $\tilde{u}$ belong to the same conjugacy class in $G(r,1,n)$, we can rewrite $(3.1)$ as the follows,
		\begin{eqnarray}
			\sum_{\substack{sig(u)=\alpha\\
					\tilde{u}<\omega}} Tr(\rho_1(\tilde{u}))
			+\sum_{\substack{sig(u)=\alpha\\
					ct(\tilde{u})=ct(\omega)}} Tr(\rho_1(\tilde{u}))\\
			=\sum_{\substack{sig(u)=\alpha\\
					\tilde{u}<\omega}} Tr(\rho_2(\tilde{u}))
			+\sum_{\substack{sig(u)=\alpha\\
					ct(\tilde{u})=ct(\omega)}} Tr(\rho_2(\tilde{u})).
		\end{eqnarray}
		If $\tilde{\omega}<\omega$ then by induction hypothesis we get
		\[Tr(\rho_1(\omega))=Tr(\rho_1(\tilde{\omega}))=Tr(\rho_2(\tilde{\omega}))=Tr(\rho_2(\omega)),\]
		so we may assume that $ct(\omega)=ct(\tilde{\omega}).$ Also, note that, if $sig(u_0)=\alpha$ and $ct(\tilde{u}_0)=ct(\omega)$, we have $ct(\tilde{u}_0)=ct(\omega)=ct(\tilde{\omega}).$ Therefore, in view of Remark 2.7 and Lemma 3.5, in all possible cases for $u$ with $sig(u)=\alpha$ and $ct(\tilde{u})=ct(\omega)$, we must have $\tilde{u}$ in the same conjugacy class as $\omega$. Therefore $Tr(\rho_i(\tilde{u}))=Tr(\rho_i(\omega))$, and \[\sum_{\substack{sig(u)=\alpha\\
				ct(\tilde{u})=ct(\omega)}} Tr(\rho_i(\tilde{u})) = c\, Tr(\rho_i(\omega)),
		\]for $i=1,2$. Since by the induction hypothesis $Tr(\rho_1(\tilde{u}))=Tr(\rho_2(\tilde{u}))$ whenever $\tilde{u}<\omega$, (3.2) reduces to \[c\, Tr(\rho_1(\omega))=c \, Tr(\rho_2(\omega)).\]
      Therefore, the theorem holds.
	\end{proof}

	\subsection{Characteristic polynomials for $G(r_1,1,n_1)\times G(r_2,1,n_2)$ }
	\hspace{1.3em}In this section, we will prove the conclusions in Theorem 3.6 hold for $G=G(r_1,1,n_1)\times G(r_2,1,n_2)$. To simplify notation, we denote $G_i=G(r_i,1,n_i)$, for $i=1,2$.
	\begin{defn} Let $G=G(r_1,1,n_1)\times G(r_2,1,n_2)$. Generating sets of $G_1,G_2,G$ are $S,T,S\bigcup T$ respectively as in Definition 2.1. Sets $S=\{s_1,t_1,$ $\cdots, s_{n_1-1},t_{n_1-1},t_{n_1}\}$ and $T=\{s_1^\prime,t_1^\prime,\cdots,$ $s_{n_2-1}^\prime,t_{n_2-1}^\prime,t_{n_2}^\prime\}$. We say a word $\omega \in G$ is in \textit{separation form} if $\omega=\omega_1 \omega_2$ such that $\omega_1 \in M(S)$ and $\omega_2 \in M(T)$. Let word $\omega \in G$, replace the letters in $\omega$ from $T$ with $1$, let $\omega_1$ be the new word obtained, replace the letters in $\omega$ from $S$ with $1$, let $\omega_2$ be the new word obtained. We call $\omega^\prime=\omega_1 \omega_2$ the \textit{separation form} of $\omega$, where $\omega_1 \in M(S)$ and $\omega_2 \in M(T)$.
	\end{defn}
	\begin{defn} Let $\omega$ be a word in $G$. The \textit{signature} of $\omega$ is $sig(\omega)=(\alpha,\beta)$ where $\alpha=(a_1,b_1,\cdots,a_{n_1-1},b_{n_1-1},b_{n_1})$ and $\beta=(a_1^\prime,b_1^\prime,\cdots,a_{n_2-1}^\prime,b_{n_2-1}^\prime,b_{n_2}^\prime)$.
		Let \[ct(u)\triangleq(\lvert u \rvert,a_1,b_1,\cdots,a_{n_1-1},b_{n_1-1},b_{n_1},a_1^\prime,b_1^\prime,\cdots,a_{n_2-1}^\prime,b_{n_2-1}^\prime)\]
		Define an order on the set $M(S \bigcup T)$ of all the words of $G$ as \[u_1\le u_2 \quad \text{if} \quad ct(u_1)<_{lex} ct(u_2) \ \text{or} \  u_1=u_2\]
		where $\le_{lex}$ is the lexicographic order on $\mathbb{N}^{2(n_1+n_2)-2}$. Clearly $(M(S \bigcup T),\le)$ has a minimum element $1$.
	\end{defn}
	\begin{rem} Since the elements in $G(r_1,1,n_1)$ commute with elements in $G(r_2,1,n_2)$, any word can be transformed into its separation form through commuting transformations. More precisely, if $\omega$ is a word in $G$, a sequence of commuting transformations yields $\omega \leadsto \omega^\prime$ where $\omega^\prime=\omega_1 \omega_2$ is the separation form of $\omega$. Note that $\omega^\prime$ is conjugate to $\omega$ and $ct(\omega)=ct(\omega^\prime)$. If $sig(\omega)=(\alpha,\beta)$,  $sig(\omega^\prime)=(\alpha,\beta)$. Also, note that, if $sig(\omega)=(\alpha,\beta)$, $sig(\omega_1)=\alpha$ and $sig(\omega_2)=\beta$.
	\end{rem}
	\begin{thm}\label{3.10} Let $\omega$ be a word and $\omega^\prime=\omega_1 \omega_2$ be the separation form of $\omega$. By Theorem 3.2, there exist words $\tilde{\omega}_1$ and $\tilde{\omega}_2$ in echelon form such that $\omega_1 \leadsto \tilde{\omega}_1$ and $\omega_2 \leadsto \tilde{\omega}_2$. Then we have $\omega \leadsto \tilde{\omega}_1 \tilde{\omega}_2$ where $\omega$ is conjugate to $\tilde{\omega}_1 \tilde{\omega}_2$ and $ct(\omega) \ge_{lex} ct(\tilde{\omega}_1 \tilde{\omega}_2)$.
	\end{thm}
	\begin{proof} If $\omega_1$ is transformed into $\omega_1^\prime$ through a sequence of commuting, canceling and replacement transformations, then the same sequence of transformations yields $\omega_1 \omega_2 \leadsto \omega_1^\prime \omega_2$. Note that if $\delta_{i_1} \cdots \delta_{i_k}\delta_{i_{k+1}} \cdots \delta_{i_N} \omega_2$ is transformed into $ \delta_{i_{k+1}} \cdots \delta_{i_N} \omega_2 \delta_{i_{1}} \cdots \delta_{i_k}$ through a circular transformation, then $ \delta_{i_{k+1}} \cdots \delta_{i_N} \omega_2 \delta_{i_{1}} \cdots \delta_{i_k}$ can be transformed into  $ \delta_{i_{k+1}} \cdots \delta_{i_N} \delta_{i_{1}} \cdots \delta_{i_k} \omega_2$ through a sequence of commuting transformations. Thus, if $\omega_1$ is transformed into $\omega_1^{\prime \prime}$ through a sequence of circular transformations, then $\omega_1 \omega_2$ can be transformed into $\omega_1^{\prime \prime} \omega_2$ through a sequence of circular and commuting transformations. Therefore, through a sequence of admissible transformations, we get $\omega_1 \omega_2 \leadsto \tilde{\omega}_1 \omega_2$. Similarly, we have $\tilde{\omega}_1 \omega_2 \leadsto \tilde{\omega}_1 \tilde{\omega}_2$. By Remark 3.9, we have $\omega \leadsto \omega_1 \omega_2 $. Thus $\omega \leadsto \tilde{\omega}_1 \tilde{\omega}_2$. Analyzing the aforementioned process, it can be known that the word $\omega$ is conjugate to $\tilde{\omega}_1 \tilde{\omega}_2 $ and $ct(\omega) \ge_{lex} ct(\tilde{\omega}_1 \tilde{\omega}_2)$.
	\end{proof}
	\begin{rem} Let $\tilde{\omega}_1$ and $\tilde{\omega}_1^\prime$ be in echelon form within $G(r_1,1,n_1)$, $\tilde{\omega}_2$ and $\tilde{\omega}_2^\prime$ be in echelon form within $G(r_2,1,n_2)$. Note that $\tilde{\omega}_1 \tilde{\omega}_2$, $\tilde{\omega}_1^\prime \tilde{\omega}_2^\prime \in G$, they can be recovered  from
		$ct(\tilde{\omega}_1 \tilde{\omega}_2)$ and
		$ct(\tilde{\omega}_1^\prime \tilde{\omega}_2^\prime)$, respectively.
	\end{rem}
	\begin{thm} Let $G=G(r_1,1,n_1)\times G(r_2,1,n_2)$. Generating sets of $G_1,G_2,G$ are $S,T$,and $S\bigcup T$, respectively, as in Definition 2.1. Sets $S=\{s_1,t_1,$ $\cdots,s_{n_1-1},t_{n_1-1},t_{n_1}\}$ and $T=\{s_1^\prime,t_1^\prime,\cdots,$ $s_{n_2-1}^\prime,t_{n_2-1}^\prime,t_{n_2}^\prime\}$. For two finite dimensional complex linear representations $\rho_1$ and $\rho_2$ of $G$, if $$d(S \bigcup T,\rho_1)=d(S \bigcup T,\rho_2),$$ we have $\rho_1\cong \rho_2$.
	\end{thm}
	\begin{proof} If $d(S \bigcup T,\rho_1)=d(S \bigcup T,\rho_2)$, by Theorem 2.2, we have
		\begin{equation}
			\sum_{sig(\omega)=(\alpha,\beta)} \chi_{\rho_1}(\omega)=\sum_{sig(\omega)=(\alpha,\beta)} \chi_{\rho_2}(\omega)
		\end{equation}
		for any $(\alpha,\beta) \in \mathbb{N}^{2(n_1+n_2)-2}$.
		
		Similar to the proof of the Theorem 3.6, it is suffice to show by induction on the partial ordering of words that for each word $\omega$ in $G$ we have
		\[Tr(\rho_1(\omega))=Tr(\rho_2(\omega)).\]
		When $\omega=1$, $sig(\omega)=(0,\cdots,0) \in \mathbb{N}^{2(n_1+n_2)-2}$. Then we have
		\[Tr(\rho_1(1))=\sum_{sig(u)=(0,\cdots,0)} \chi_{\rho_1}(u)=\sum_{sig(u)=(0,\cdots,0)} \chi_{\rho_2}(u)=Tr(\rho_2(1)).\]
		Suppose $\omega\neq1$ and we have proved our statement for all words lower than $\omega$. Let $sig(\omega)=(\alpha,\beta)$. For each word $u$ with $sig(u)=(\alpha,\beta)$, let $\tilde{u}_1 \tilde{u}_2$ be a word such that $u \leadsto \tilde{u}_1\tilde{u}_2 $, according to Theorem \ref{3.10}. Thus, by Theorem 3.10, we have $ct(\omega)=ct(u) \ge_{lex} ct(\tilde{u}_1 \tilde{u}_2)$. Let
		\[c=\lvert \{u\ |sig(u)=(\alpha,\beta) \  \text{and}\   ct(\tilde{u}_1 \tilde{u}_2)=ct(\omega)\}\rvert .\]
		Since $u$ and $\tilde{u}_1 \tilde{u}_2$ belong to the same conjugacy class in $G$, we can rewrite $(3.3)$ as follows:
		\begin{eqnarray}\label{3.4}
			\sum_{\substack{sig(u)=(\alpha,\beta)\\
					\tilde{u}_1 \tilde{u}_2<\omega}} Tr(\rho_1(\tilde{u}_1 \tilde{u}_2))
			+\sum_{\substack{sig(u)=(\alpha,\beta)\\
					ct(\tilde{u}_1 \tilde{u}_2)=ct(\omega)}} Tr(\rho_1(\tilde{u}_1 \tilde{u}_2))\\
			=\sum_{\substack{sig(u)=(\alpha,\beta)\\
					\tilde{u}_1 \tilde{u}_2<\omega}} Tr(\rho_2(\tilde{u}_1 \tilde{u}_2))
			+\sum_{\substack{sig(u)=(\alpha,\beta)\\
					ct(\tilde{u}_1 \tilde{u}_2)=ct(\omega)}} Tr(\rho_2(\tilde{u}_1 \tilde{u}_2)).
		\end{eqnarray}
		If $\tilde{\omega}_1 \tilde{\omega}_2<\omega$, by induction hypothesis, it follows that
		\[Tr(\rho_1(\omega))=Tr(\rho_1(\tilde{\omega}_1 \tilde{\omega}_2))=Tr(\rho_2(\tilde{\omega}_1 \tilde{\omega}_2))=Tr(\rho_2(\omega)),\]
		so we may assume that $ct(\omega)=ct(\tilde{\omega}_1 \tilde{\omega}_2).$ Also, we see that $ct(\tilde{u}_1^\prime \tilde{u}_2^\prime)=ct(\omega)=ct(\tilde{\omega}_1 \tilde{\omega}_2),$ if $sig(u^\prime)=(\alpha,\beta)$ and $ct(\tilde{u}_1^\prime \tilde{u}_2^\prime)=ct(\omega)$.  Therefore, by Remark 3.11, in all possible cases for $u$ with $sig(u)=(\alpha,\beta)$ and $ct(\tilde{u}_1 \tilde{u}_2)=ct(\omega)$ we must have $\tilde{u}_1 \tilde{u}_2=\tilde{\omega}_1 \tilde{\omega}_2$ and $\tilde{u}_1 \tilde{u}_2$ in the same conjugacy class as $\omega$. Therefore, we have \[\sum_{\substack{sig(u)=(\alpha,\beta)\\
				ct(\tilde{u}_1 \tilde{u}_2)=ct(\omega)}} Tr(\rho_i(\tilde{u}_1 \tilde{u}_2)) = c\, Tr(\rho_i(\omega)),
		\]for $i=1,2$. Since by the induction hypothesis $Tr(\rho_1(\tilde{u}_1 \tilde{u}_2))=Tr(\rho_2(\tilde{u}_1 \tilde{u}_2))$ whenever $\tilde{u}_1 \tilde{u}_2<\omega$, the formula (\ref{3.4}) reduces to \[c\, Tr(\rho_1(\omega))=c \, Tr(\rho_2(\omega)),\] and the desired conclusion holds immediate.
	\end{proof}
	If $\rho_1\cong \rho_2$, it is easy to see that $d(S \bigcup T,\rho_1)=d(S \bigcup T,\rho_2)$. Consequently, we have:
	\begin{cor} Let $G=G(r_1,1,n_1)\times G(r_2,1,n_2)$. Generating sets of $G_1,G_2,G$ are $S,T,S\bigcup T$ respectively as in Definition 2.1. For two finite dimensional complex linear representations $\rho_1$ and $\rho_2$ of $G$,  $d(S \bigcup T,\rho_1)=d(S \bigcup T,\rho_2)$ if and only if $\rho_1\cong \rho_2$.
	\end{cor}
	Similar to by induction $G=G(r_1,1,n_1)\times G(r_2,1,n_2)$, we can obtain the following corollary.
	
	\begin{cor} Let $G=G(r_1,1,n_1) \times \cdots \times G(r_k,1,n_k)$. Generating sets of $G_1,\cdots,G_k,G$ are $S_1,\cdots,S_k$, and $S_1\bigcup \cdots \bigcup S_k$, respectively. For two finite dimensional complex linear representations $\rho_1$ and $\rho_2$ of $G$,  $d(S_1\bigcup \cdots \bigcup S_k,\rho_1)$ $=d(S_1\bigcup \cdots \bigcup S_k,\rho_2)$ if and only if $\rho_1\cong \rho_2$.
	\end{cor}

	\section{Characteristic polynomials for affine Coxeter groups and $G(r,p,n)$}
%	In this chapter, we will seek a suitable subset $S$ of the affine Coxeter group to define the characteristic polynomial and prove that for every finite dimensional representation of  affine Coxeter group $\widetilde{W}$, the characteristic polynomial of $\widetilde{W}$ determines the character of this representation.

\subsection{Characteristic polynomials for Affine coxeter groups}
	
	\hspace{1.3em}We begin with reviewing key results associated to affine Coxeter groups and basic concepts in the representation theory of finite groups which can be found in \cite{M1998, S1977}.
	
\begin{defn} A Coxeter group $W$ is generated by a subset $S$ and subjected to the following relations; \[W=\langle s \in S \lvert (ss^\prime)^{m_{ss^\prime}}=1 \rangle,\]where
    \[m_{ss}=1 \,\, \text{and} \,\, m_{ss^\prime}\in \lbrace 2,3,\cdots \rbrace \cup \lbrace \infty \rbrace \,\,\, \text{if} \,\, s \neq s^\prime.\]
	\end{defn}
In particular, the relations $m_{ss^\prime}=1$ asserts that each $s$ is an involution. The pair $(W,S)$ is called a \textit{Coxeter system}.
\begin{rem} An important class of Coxeter groups are the \textit{affine Coxeter groups}. A Coxeter group is called affine, if it can be represented as a reflection group on $\mathbb{R}^n$ with no common invariant subspace with dimension less than $n$. Every affine Coxeter group arises from a finite Coxeter group in the following way: Let $\Phi$ be a crystallographic root system for a finite Coxeter group $W$ (which exists exactly for the so-called finite Weyl groups). We define an affine reflection as follows:
\[s_{\alpha,k}(v):=v- ( \langle v,\alpha \rangle -k ) \dfrac{2\alpha}{\langle \alpha,\alpha \rangle} \,\,\, \text{for}\,\, k\in \mathbb{R},\, \alpha \in \Phi.\]
This is indeed a reflection on the affine hyperplane \[H_{\alpha,k}:=\lbrace v \in V \lvert \langle v,\alpha \rangle =k \rbrace.\] The group generated by all those reflections(which include the reflections $s_\alpha=s_{\alpha,0}$ of $W$ ) is denoted by $\widetilde{W}$. Furthermore, if we start with an \textit{irreducible} finite Coxeter system $(W,S)$ (meaning that the Coxeter diagram of $W$ is connected), we can find an \textit{irreducible} root system $\Phi$ for $W$, i.e. there is no partition of $\Phi$ into root systems such that any two roots in different subsets are orthogonal. Given an irreducible root system, there is a root $\widetilde{\alpha}$ such that for any root $\alpha$, the root $\widetilde{\alpha}-\alpha$ is a sum of simple roots. This root is unique and we call it the \textit{highest root}. The affine group $\widetilde{W}$ is generated by $\widetilde{S}=S \cup \{s_{\widetilde{\alpha},1}\}$ and $(\widetilde{W},\widetilde{S})$ is the corresponding affine Coxeter system.
\end{rem}
Now we talk about \textit{words}
within a given group $G$ as discussed in Section 2. Let $S=\{s_1,\cdots,s_n\}$ be a subset of $G$. $M(S)$ is a monoid generated by alphabet $S$. There is a natural homomorphism
\[\phi:M(S) \longrightarrow G.\]
For a word $\omega \in M(S)$, the \textit{signature} of $\omega$ is \[sig(\omega)=(a_1,\cdots,a_n) \in \mathbb{N}^{n}.\]
For any $\rho \in rep(G)$, the \textit{characteristic polynomial} of $\rho$ is defined as
 \[d(S,\rho)=det(x_0 I+\sum_{i=1}^n x_i \rho(s_i)).\]

Analogous to Theorem 2.4, we have the following lemma.
    \begin{lemma}Let $G$ be a affine Coxeter group. Let $\rho_1,\rho_2 \in rep(G)$. If $d(S,\rho_1)=d(S,\rho_2)$, then for each given $\alpha \in \mathbb{N}^{n}$, we have\[\sum_{sig(\omega)=\alpha} \chi_{\rho_1}(\omega)=\sum_{sig(\omega)=\alpha} \chi_{\rho_2}(\omega),\]
\textit{where $\chi_{\rho_i}$ denotes the character of representation $\rho_i$, for $i=1,2$.}
\end{lemma}
   % \subsection{Characteristic polynomials for Affine coxeter groups}
	\hspace{1.3em} We will seek a suitable subset $S$ of the affine Coxeter group to define the characteristic polynomial and prove that for every finite dimensional representation of  affine Coxeter group $\widetilde{W}$, the characteristic polynomial of $\widetilde{W}$ determines the character of this representation. Before proving these, we first prove a theorem, which plays a crucial role in proving our main conclusion. In the following theorem, we consider an arbitrary group $G=H\ltimes N$, where $N$ is abelian. Here $\omega_1 \thicksim \omega_2$ means that $\omega_1$ is conjugate to $\omega_2$.
\begin{thm}Let $G=H\ltimes N$, where $N$ is abelian. Let $\{n_1,\cdots,n_k\}$ be a set of generators of N, $\{h_1,\cdots,h_t\}$ be a set which represents every conjugacy class of H, and let the set $S=\{h_1,\cdots,h_t,n_1,\cdots,n_k, n_1^{-1},\cdots,n_k^{-1}\}$. Then, for two finite dimensional complex linear representations $\rho_1$ and $\rho_2$ of $G$, $d(S,\rho_1)=d(S,\rho_2)$ if and only if $\rho_1\cong \rho_2$.
\end{thm}
\begin{proof}
If two representations $(\rho_1,V_1)$ and $(\rho_2,V_2)$ of a group $G$ are called equivalent, then there is an isomorphism $X:V_1 \rightarrow V_2$ such that for all $\omega \in G$ we have \[\rho_1(\omega)=X^{-1} \rho_2(\omega)X.\]
From this, it is easy to see that $d(S,\rho_1)=d(S,\rho_2)$.

If $d(S,\rho_1)=d(S,\rho_2)$, according to Lemma 4.3, we have
\begin{equation}
    \sum_{sig(\omega)=\alpha} \chi_{\rho_1}(\omega)=\sum_{sig(\omega)=\alpha} \chi_{\rho_2}(\omega).
\end{equation}
Since $G=H\ltimes N$, it is known that $G/N \cong H$. The set $\{h_1,\cdots,h_t\}$ represents every conjugacy classes of $H$. Thus, $\{h_1 N,\cdots,h_t N\}$ be a set which represents every conjugacy classes of $G/N$. Let $\omega \in G$, then there must exist an $i \in [t]$ such that $\omega N$ and $h_i N$ are conjugate. Then there exists a $u\in G$ such that \[\omega N= (u N)^{-1} h_i N (u N)=u^{-1}h_i u N.\]
Therefore, there exists a $v \in N$ such that\[\omega=u^{-1}h_i u v \thicksim h_i u v u^{-1}.\]
Since $N$ is a normal abelian subgroup, \[uvu^{-1}=n_1^{b_1}\cdots n_k^{b_k} (n_1^{-1})^{c_1} \cdots (n_k^{-1})^{c_k}.\]
(Note that for the same $j$, $b_j$ and $c_j$ cannot both be non-zero at the same time; otherwise, there would be infinitely many values of $b_j$ and $c_j$ that satisfy the equation above.)

Let \[\alpha=sig(h_in_1^{b_1}\cdots n_k^{b_k} (n_1^{-1})^{c_1} \cdots (n_k^{-1})^{c_k}),\  E_{\alpha}=\{u \in M(S):sig(u)=\alpha\}.\] Select a word from $E_{\alpha}$. This word has form $\omega_1h_i \omega_2$, where $\omega_1$ and $\omega_2$ are words comprised of letters $n_j$ and $n_j^{-1}$, $j=1,\cdots,k$ with the total number of $n_j$ in both words being equal to $b_j$ and the total number of $n_j^{-1}$ being equal to $c_j$. Now we see that \[\omega_1 h_i \omega_2 \thicksim h_i \omega_2 \omega_1=h_in_1^{b_1}\cdots n_k^{b_k} (n_1^{-1})^{c_1} \cdots (n_k^{-1})^{c_k},\]
where the last equality follows from the fact that $N$ is abelian. Equation (4.1) implies that\[\lvert E_{\alpha}\rvert \chi_{\rho_1}(\omega)=\lvert E_{\alpha}\rvert \chi_{\rho_2}(\omega).\]
Therefore, we have $\chi_{\rho_1}=\chi_{\rho_2}$. Furthermore, we have $\rho_1 \cong \rho_2$.
\end{proof}
\begin{rem}In fact, there is some redundancy in defining the characteristic polynomial with \(S\). Since \(h_1, \ldots, h_t\) are representatives of all conjugacy classes of $G$, one of them must satisfy \(h_i = 1\). If we now define the characteristic polynomial of $G$ using \(\widetilde{S} = S \setminus \{1\}\), and employ a similar method in the proof of the above, the conclusion of Theorem 4.4 still holds.	
\end{rem}
Let $\widetilde{W}$ be an irreducible affine Coxeter group. The affine Coxeter group $\widetilde{W}$ can be regarded as a semidirect product $W \ltimes T$ of the root lattice $T$ by the corresponding finite Coxeter group $W$. $T$ is a finitely generated normal abelian subgroup of $\widetilde{W}$, and $\{\alpha_1,\cdots,\alpha_n\}$ is its set of generators. $W$ is a finite irreducible Coxeter group, hence the number of conjugacy classes of $W$ is finite.

Let $S_1$ be a set which represents every conjugcy clss of $W$. Sets $S_2=\{\alpha_1,\cdots,\alpha_n\}$ and $S_2^{-1}=\{\alpha_1^{-1},\cdots,\alpha_n^{-1}\}$. If we take $S=S_1 \bigcup S_2 \bigcup S_2^{-1}$, then by Theorem 4.4, we can immediately derive the following corollary.
\begin{cor}Let $\widetilde{W}$ be an irreducible affine Coxeter group. For two finite dimensional complex linear representations $\rho_1$ and $\rho_2$ of $\widetilde{W}$, $d(S,\rho_1)=d(S,\rho_2)$ if and only if $\rho_1\cong \rho_2$.
\end{cor}
\begin{rem}
%For the irreducible affine Coxeter groups $\widetilde{A_n}$, $\widetilde{B_n}$ and $\widetilde{D_n}$, we can identify  specific set $S_1$.
It is well known that  $\widetilde{A_n}=W(A_n) \ltimes T$ where $W(A_n)$ is Coxeter group with respect to Lie algebra $A_n$, and \[W(A_n)=\langle s_1,\cdots, s_n  \lvert (s_i s_j)^{m_{i j}}=1 \rangle ,\]
where
    \[m_{ii}=1 \,\, \text{and} \,\, m_{i j}=3 \,\,\, \text{if} \,\, i \neq j.\]
Taking the irreducible affine Coxeter group $\widetilde{A_n}$ as an example, from \cite[Theorem 6.5]{CS2021} and \cite[Definition 6.1]{CS2021}, we can understand that each element in $A_n$ is conjugate to an element of the form $\delta_1 \cdots \delta_n$, where $\delta_{i} \in \{1,s_i\}$, $i=1,\cdots,n$. Therefore, we can define $S_1$ as the set composed of all distinct $\delta_1 \cdots \delta_n$. Similarly, for the irreducible affine Coxeter groups $\widetilde{B_n}$ and $\widetilde{D_n}$, such an $S_1$ can also be found accordingly.
\end{rem}

        \subsection{The subgroup $G(r,p,n)$ of $G(r,1,n)$ }
	\hspace{1.3em}In fact, we consider the symmetric group $\Sigma_n$ as a subgroup in $GL(n,\mathbb{C}_n)$ as those matrices with just $n $ $1$s and $n(n-1)$ $0$s as their matrix entries. And we extend it to those matrices in $GL(n,\mathbb{C}_n)$ with just $n$ nonzero entries but those nonzero entries have values of $r$th root of $1$, then we obtain the imprimitive group $G(r,1,n)$. Now we present the definition of $G(r,p,n)$.
	
	\begin{defn}Let $r,p,n$ be integers such that $p$ divides $r$. We put $d=r/p$. Then
	  \[G(r,p,n)=\Sigma_n \ltimes A(r,p,n),\]
	%where \[\Sigma_n= \text{the permutation matrices in} \,\,GL_n(\mathbb{C}),\]
	\begin{equation*}
		A(r,p,n)=\left\{\left. diag\{\omega_1,\cdots,\omega_n\}
		\right\vert
		\begin{aligned}
			\omega_i^r=1 \ \text{and}\
			(\omega_1 \cdots \omega_n)^{r/p}=1
		\end{aligned}
		\right\}.
	\end{equation*}
	\end{defn}
	\begin{prop}Let \[y_1=diag\{e^{2p\pi i/r},1,\cdots,1\},
		y_2=
		diag\{e^{-2\pi i/r},e^{2\pi i/r},1,\cdots,1\}\]
		\[\cdots,
		y_n=
		diag\{1,\cdots,1,e^{-2\pi i/r},e^{2\pi i/r}\}.
		\]Then $y_1,\cdots,y_n$ belonging to $A(r,p,n)$ and generating $A(r,p,n)$.
	\end{prop}
	\begin{proof}
		For any element \[y=
		diag\{e^{2k_1 \pi i/r},\cdots,e^{2k_n \pi i/r}\},
		\]
		in $A(r,p,n)$, we have \[0 \le k_i <r \,\, \text{and} \,\, \sum_{i=1}^n k_i=pk, 0 \le k<nr/p.\]
		Now it is easy to verify that $y=y_1^{k}y_2^{\sum_{j=2}^n k_j} \cdots y_i^{\sum_{j=i}^n k_j} \cdots y_n^{k_n}$.
	\end{proof}
	If we let
   \[M_1=
		\begin{pmatrix}
			0              & 1      & 0      & \cdots & 0 \\
			1              & 0      & 0      & \cdots & 0 \\
			0              & 0      & 1      & \cdots & 0 \\
			\vdots         & \vdots & \vdots & \ddots & \vdots \\
			0              & 0      & 0      & \cdots & 1
		\end{pmatrix},\cdots,
		M_{n-1}=
		\begin{pmatrix}
			1              & \cdots & 0      & 0      & 0 \\
			\vdots         & \ddots & \vdots & \vdots & \vdots \\
			0              & \cdots      & 1      & 0 & 0 \\
			0              & \cdots     & 0      & 0 & 1 \\
			0              & \cdots      & 0      & 1 & 0
		\end{pmatrix},
		\]
		then by correlating $M_i$ with $s_i$, we can deduce that $\Sigma_n \cong W(A_{n-1})$, where $W(A_{n-1})$ refers to the mentioned in the Remark 4.7.

	%\subsection{The characteristic polynomial for $G(r,p,n)$ }
	%\hspace{1.3em}
	In the previous section, we have already proven that $S_2=\{y_1,\cdots,y_n\}$ can serve as a generating set for $A(r,p,n)$. Let $S_1$ be a set which represents every conjugcy clss of  $\Sigma_n$. By Theorem 4.4, we can immediately derive the following theorem.

	\begin{thm} Let $S=S_1 \bigcup S_2 \bigcup S_2^{-1}$ be the generating set of $G(r,p,n)$. For two finite dimensional complex linear representations $\rho_1$ and $\rho_2$ of $G(r,p,n)$, $d(S,\rho_1)=d(S,\rho_2)$ if and only if $\rho_1\cong \rho_2$.
	\end{thm}
In fact, there is some redundancy in defining the characteristic polynomial with $S$.
\begin{cor}Let $\widetilde{S}=S_1 \bigcup S_2 $ be the generating set of $G(r,p,n)$. For two finite dimensional complex linear representations $\rho_1$ and $\rho_2$ of $G(r,p,n)$, $d(\widetilde{S},\rho_1)=d(\widetilde{S},\rho_2)$ if and only if $\rho_1\cong \rho_2$.
\end{cor}
    \begin{proof}Since the generators $y_i$ in $A(r,p,n)$ are of finite order, the elements $y_i^{-1}$ can actually be expressed as powers of $y_i$. Therefore, we can define the characteristic polynomial of $G(r,p,n)$ using $\widetilde{S}=S_1\bigcup S_2$, and a proof similar to that of Theorem 4.4 can be used to prove the corollary.
    \end{proof}
    \begin{rem}Here, we can also identify a specific $S_1$. Due to $\Sigma_n \cong W(A_{n-1})$ and Remark 4.7, we can define $S_1$ as the set composed of all distinct $\delta_1 \cdots \delta_{n-1}$, where $\delta_i \in \{1,M_i\}$ and $i=1,\cdots,n-1$.
    \end{rem}
	We find that when $p=1$, it is the group $G(r,1,n)$ discussed in the previous chapter. However, the generators here are different from those in the previous section. Although one can quickly prove the conclusions of the previous section using the generators of this section, the generators used in the previous section are more standard and have greater value.
	\begin{rem}
	For complex reflection groups, the characteristic polynomials vary depending on the choice of bases. Consequently, a pertinent question arises: how many generating elements are required to fully capture the  information of representations of the imprimitive group? The representations theory of groups is intimately connected with algebraic geometry. This paper establishes the correspondence between the characteristic polynomial and the representations of imprimitive group, raising the possibility of deriving unexpected results from an algebraic geometry perspective. Further exploration will reveal these outcomes.
		\end{rem}

	Chenyue Feng\\
	Email: fcy1572565124@163.com\\
	School of Mathematics, Shandong University\\
	Shanda Nanlu 27, Jinan, \\
	Shandong Province, China\\
	Postcode: 250100\\
	Shoumin Liu\\
	Email: s.liu@sdu.edu.cn\\
	School of Mathematics, Shandong University\\
	Shanda Nanlu 27, Jinan, \\
	Shandong Province, China\\
	Postcode: 250100\\
	Xumin Wang\\
	Email: 202320303@mail.sdu.edu.cn\\
	School of Mathematics, Shandong University\\
	Shanda Nanlu 27, Jinan, \\
	Shandong Province, China\\
	Postcode: 250100
	

\begin{thebibliography}{99}
	
	\bibitem{BB2005}
	  A. Bj?rner, F. Brenti, Combinatorics of Coxeter Groups, Graduate Texts in Mathematics, Vol.231, Springer, New York, (2005).
	
	
		\bibitem{CCD2019}
		Z. Chen, X. Chen, M. Ding, On the characteristic polynomial of $\mathfrak{sl}(2, \C)$, Linear Alg. Appl. 579 (2019), 237--243.
		%MR3959733, Zbl:1462.17011, doi:10.1016/j.laa.2019.05.036
		
	
		
		\bibitem{Cu2000}
		C. W. Curtis, Pioneers of representation theory, Bull. Amer. Math. Soc. 37 (2000), no. 3, 359-362. %MR1715145, Zbl:0939.01007, doi:10.1090/hmath/015
		
		\bibitem{CS2021}
		 ?. ?u?kovi?, M. I. Stessin, A. B. Tchernev, Determinantal hypersurfaces and representations of Coxeter groups, Pacific Journal of Mathematics, 313(1)(2021), 103-135.
		
		
		\bibitem{D1969}
		R. Dedekind, Gesammelte Mathematische Werke, Vol. II, Chelsea, New York, (1969). %doi:10.1017/S002555720020750X
		
		\bibitem{Di1902}
		A. C. Dixon, Note on the reduction of a ternary quantic to a symmetrical determinan, Cambr. Proc. 11 (1902), no. 5, 350-351.
		
		\bibitem{Di1921}
		L. E. Dickson, Determination of all general homogeneous polynomials expressible as determinants with linear elements,
		Trans. Amer. Math. Soc. 22 (1921), no. 2, 167-179.
		%MR1501168, Zbl:48.0099.02, doi:10.1090/S0002-9947-1921-1501168-0
		
		\bibitem{Di1975}
		L. E. Dickson, An elementary exposition of Frobenius theory of group characters and group-determinants, Ann. Math. 4 (1902), 25-49; Mathematical Papers, Vol. II. Chelsea, New York, 1975, 737-761.
		%MR1502293, Zbl:33.0149.03, doi:10.2307/1967149
		
		
		\bibitem{FLW2024}
		 C. Feng, S. Liu ,  X. Wang . Characteristic polynomials for classical Lie algebras, arXiv preprint arXiv:2410.19354, (2024).
		
		
		\bibitem{FS1991}
		 E. Formanek, D Sibley, The Group Determinant Determines the Group, Proc. of
		AMS, 112(1991), 649-656.
		
		
		\bibitem{F1896}
		F. G. Frobenius, \"uber vertauschbare Matrizen, Sitzungsberichte der K\"oniglich Preussischen,Akademie der Wissenschaften zu Berlin, (1896), 601--614.
		%Zbl:27.0109.04  doi:10.1007/BF03016759
		
		
		\bibitem{GLW2024} A. Geng, S. Liu, X. Wang, Characteristic polynomials and finite dimensional representations of simple Lie algebras, New York J. Math. 30 (2024), 24--37.
		
	
		
		\bibitem{GY2017} R. Grigorchuk, R. Yang, Joint spectrum and the infinite dihedral groups, Proc. Steklov Inst.Math. 297(1) (2017), 145--178.
		%MR3695412, Zbl:1462.47003, arXiv:1605.01547, doi:10.1134/S0081543817040095
		
		\bibitem{GR2014}
		W. He, R. Yang, Projective spectrum and kernel bundle, Sci. China Math. 57 (2014), 1--10.
		%MR3426136, Zbl:1334.47008, doi:10.1007/s11425-015-5043-z
		
		\bibitem{HY2023}
		W. He, R. Yang, On the unitary equivalence of compact operator tuples, Sci. China. Math. 66 (2023), 571--580.
		
		\bibitem{HKV2018}
		J. W. Helton, I. Klep, J. Volcic,
		Geometry of free loci and factorisation of noncommutative polynomials
		Adv. Math., 331 (2018), 589-626.
		
		
		
		\bibitem{HY2024}
		K. Howell, R. Yang, The characteristic polynomial of projections, Lin. Alg. Appl. 680
		(2024), 170--182.
		
		
		\bibitem{HY2018}
		Z. Hu, R. Yang, On the charactersitic polynomials of multiparameter pencils,  Linear Alg. Appl. 558 (2018), 250--263.
		%MR3854197, Zbl:1398.15006, doi:10.1016/j.laa.2018.08.036
		
		\bibitem{HZ2019}
		Z. Hu, P. Zhang, Determinant and charactreristic polynomials of Lie algebras, Linear Alg. Appl. 563 (2019), 426--439.
		%MR3880055, Zbl:1405.17018, doi:10.1016/j.laa.2018.11.015
		
		\bibitem{H2021}
		Z. Hu, Eigenvalues and eigenvectors of a class of irreducible tridiagonal matrices, Linear Alg. Appl. 619 (2021), 328--337.
		%MR4231560, Zbl:1495.15013, doi:10.1016/j.laa.2021.03.014
		
	
	
	
		
		\bibitem{JL2022}
		T. Jiang, S. Liu, Characteristic polynomials and finitely dimensional representations of $\mathfrak{sl}(2, \C)$, Linear Alg. Appl. 647 (2022), 78--88.
		%MR4413328, Zbl:1498.17012, arXiv:2112.07933, doi:10.1016/j.laa.2022.04.008
		
		\bibitem{C1990}
		Conway, J.B,  A Course in Functional Analysis, Springer, New York (1990).
		
		
		\bibitem{KV2012}
		D. Kerner, V. Vinnikov,
		Determinantal representations of singular hypersurfaces in $\mathbb{P}^n$,
		Adv. Math., 231 (2012), 1619-1654.
		
		
		
		\bibitem{KKSW}K. Korkeathikhun, B. Khuhirun, S. Sriwongsa, K. Wiboonton, More on characteristic polynomials of Lie algebras, J. Algebra. 643 (2024), 294--310.
		
		\bibitem{RK2001}
		Richard. M. Kane, Reflection groups and invariant theory, Vol. 5, Springer, New York, (2001).
		
		
		\bibitem{KY2021}
		F. Azari Key, R. Yang, Spectral invariants for finite dimensional Lie algebras, Linear Alg. Appl. 611 (2021), 148--170.
		%MR4190618, Zbl:1484.17005, arXiv:2004.00551, doi:10.1016/j.laa.2020.10.024
		
			\bibitem{M1998}	
		G. Maxwell, The normal subgroups of finite and affine Coxeter groups, Proceedings of the London Mathematical Society, 76(2)(1998), 359-382.
		
		\bibitem{L2017}
		S. Liu, Clifford theory on rational Cherednik algebras of imprimitive groups, Indag. Math (N.S.), 28 (2017), no. 4, 736--748.
		
		\bibitem{PS2024}
		 T. Peebles,  M. I. Stessin, Projective joint spectra and characters of representations of $\tilde{A}_n$, Journal of Mathematical Analysis and Applications, 532(1)(2024): 127951.
		
		
		\bibitem{S1977}
		J. P. Serre, Linear Representations of Finite Groups,
		Graduate Texts in Mathematics, vol. 42, Springer-Verlag, New York-Heidelberg, (1977).
	
		\bibitem{S2024}
		Stessin M I. Spectral reconstruction of operator tuples, Advances in Operator Theory, 9(4)(2024), 80.
		
		
		\bibitem{Y2024}
		R. Yang, A spectral theory of noncommuting operators, Springer, (2024).
		
		\bibitem{Y2009}
		R. Yang, Projective spectrum in Banach algebras, J. Topo. Anal. 1(3) (2009), 289--306.
		%MR2574027, Zbl:1197.47015, arXiv:0804.0387, doi:10.1142/S1793525309000126\\
	\end{thebibliography}
\end{document}